\documentclass[12pt]{amsart}
\usepackage{a4wide,color,graphicx}
\usepackage{amsmath}
\usepackage{enumerate}
\allowdisplaybreaks

\let\pa\partial
\let\na\nabla
\let\eps\varepsilon
\newcommand{\N}{{\mathbb N}}
\newcommand{\R}{{\mathbb R}}
\newcommand{\diver}{\operatorname{div}}

\newcommand{\E}{{\mathbb E}}
\renewcommand{\d}{\mathrm{d}}
\newcommand{\red}{\textcolor{black}}

\newtheorem{theorem}{Theorem}
\newtheorem{lemma}[theorem]{Lemma}
\newtheorem{proposition}[theorem]{Proposition}
\newtheorem{remark}[theorem]{Remark}

\begin{document}  

\title[Population cross-diffusion systems]{Rigorous derivation of \\
population cross-diffusion systems from \\
moderately interacting particle systems}

\author[L. Chen]{Li Chen}
\address{University of Mannheim, School of Business Informatics and Mathematics, 
68131 Mann\-heim, Germany}
\email{chen@math.uni-mannheim.de}

\author[E. S. Daus]{Esther S. Daus}
\address{Institute of Analysis and Scientific Computing, TU Wien, 
Wiedner Hauptstra\ss e 8--10, 1040 Wien, Austria}
\email{esther.daus@tuwien.ac.at} 

\author[A. Holzinger]{Alexandra Holzinger}
\address{Institute of Analysis and Scientific Computing, TU Wien, 
Wiedner Hauptstra\ss e 8--10, 1040 Wien, Austria}
\email{alexandra.holzinger@tuwien.ac.at}

\author[A. J\"ungel]{Ansgar J\"ungel$^*$}
\address{Institute of Analysis and Scientific Computing, TU Wien, 
Wiedner Hauptstra\ss e 8--10, 1040 Wien, Austria {\rm (${}^*$Corresponding author)}}
\email{juengel@tuwien.ac.at} 

\date{\today}

\thanks{The first author acknowledges support from the DFG project CH955/3-1. 
The last three authors acknowledge partial support from   
the Austrian Science Fund (FWF), grants F65, I3401, P30000, P33010, and W1245.} 

\begin{abstract}
Population cross-diffusion systems of Shigesada--Kawasaki--Tera\-moto type 
are derived in a mean-field-type limit from stochastic, moderately interacting
many-particle systems for multiple population species in the whole space.
The diffusion term in the stochastic model depends nonlinearly 
on the interactions between the
individuals, and the drift term is the gradient of the environmental potential.
In the first step, the mean-field limit leads to an intermediate nonlocal model.
The local cross-diffusion system is derived in the second step in a moderate 
scaling regime, when the interaction potentials approach the Dirac delta distribution.
The global existence of strong solutions to the intermediate and the local diffusion
systems is proved for sufficiently small initial data. Furthermore, 
numerical simulations on the particle level are presented.
\end{abstract}



\maketitle


\section{Introduction}

The aim of this paper is to derive the population cross-diffusion system of
Shigesada, Kawasaki, and Teramoto \cite{SKT79} from a stochastic,
moderately interacting particle system in a mean-field-type limit. 
More precisely, we derive the system of equations
\begin{equation}\label{1.skt}
  \pa_t u_i = \diver(u_i\na U_i) + \Delta\bigg(\sigma_i u_i
	+ u_i\sum_{j=1}^n f(a_{ij}u_j)\bigg),\quad
	u_i(0) = u_{0,i}\quad\mbox{in }\R^d,\ t>0,
\end{equation}
where $i=1,\ldots,n$ is the species index, \red{$d\ge 1$ the space dimension},
$u=(u_1,\ldots,u_n)$ is the vector of population densities, and
$U_i=U_i(x)$ are given environmental potentials. The parameters $\sigma_i>0$
are the constant diffusion coefficients in the stochastic system, and
$a_{ij}\ge 0$ are limiting values of the interaction potentials.
In the linear case $f(s)=s$, we obtain the population 
model in \cite{SKT79}. System \eqref{1.skt} with nonlinear functions $f$
have also been studied in the \red{mathematical} literature; 
see, e.g., \cite{CDJ18,DLMT15,LeMo17}.
\red{Such systems can be formally derived from random walks on a lattice, where
the nonlinearity originates from the transition rates in the random-walk model
\cite[Appendix A]{ZaJu17}. Assuming that the transition rates depend in a nonlinear
way on the densities leads to equations similar to \eqref{1.skt}.}
We assume that $f$ is smooth but possibly {\em not} globally Lipschitz continuous
(including power functions).
Our results are valid for functions $f_i$ depending on the species type, but
we choose the same function for all species to simplify the presentation.

This paper extends the many-particle limit of \cite{CDJ19} leading to the
cross-diffusion system
\begin{equation}\label{1.cdj}
  \pa_i u_i = \diver\bigg(\sigma_i\na u_i + \sum_{j=1}^n a_{ij}u_i\na u_j\bigg)
	\quad\mbox{in }\R^d,\ t>0,\ i=1,\ldots,n,
\end{equation}
which differs from \eqref{1.skt} by the drift term, the nonlinear function $f$,
and the diffusion term $\diver\sum_{j=1}^n a_{ij}u_j\na u_i$. System \eqref{1.cdj}
is the mean-field limit of the particle system for $N$ individuals
\begin{equation}\label{1.cdj2}
\begin{aligned}
  & \d Y_{k,i}^{N,\eta} = -\sum_{j=1}^n\frac{1}{N}\sum_{\ell=1}^N
	\na B_{ij}^\eta\big(Y_{k,i}^{N,\eta}-Y_{\ell,j}^{N,\eta}\big)\d t
	+ \sqrt{2\sigma_i}\d W_i^k(t), \\ 
	& Y_{k,i}^{N,\eta}(0)=\xi_i^k, \quad i=1,\ldots,n,\ k=1,\ldots,N,
\end{aligned}
\end{equation}
where $(W_i^k(t))_{t\ge 0}$ are 
$d$-dimensional Brownian motions and $\xi_i^1,\ldots,\xi_i^N$ are independent 
and identically distributed (iid) random variables with the common probability 
density function $u_{0,i}$. The functions 
\begin{equation}\label{1.scaling.B}
  B_{ij}^\eta(x) = \eta^{-d}B_{ij}\bigg(\frac{|x|}{\eta}\bigg), \quad x\in\R^d,
\end{equation}
are interaction
potentials regularizing the delta distribution $\delta_0$, i.e.\
$B_{ij}^\eta\to a_{ij}\delta_0$ as $\eta\to 0$ in the sense of distributions.

System \eqref{1.skt} is derived from an interacting particle system for
$n$ species with particle numbers $N_1,\ldots,N_n$, 
moving in the whole space $\R^d$. To simplify, we set $N=N_i$ for all $i=1,\ldots,n$.
The key idea of this paper is to consider interacting diffusion coefficients:
\begin{equation}\label{1.pl}
\begin{aligned}
  & \d X_{k,i}^{N,\eta} = -\na U_i(X_{k,i}^{N,\eta})\d t
	+ \bigg(2\sigma_i + 2\sum_{j=1}^nf_\eta\bigg(\frac{1}{N}\!\!\!\!
	\sum_{\substack{\ell=1\\ (\ell,j)\neq (k,i)}}^N\!\!\!\!
	B_{ij}^\eta(X_{k,i}^{N,\eta}-X_{\ell,j}^{N,\eta})\bigg)\bigg)^{1/2}\d W_i^k(t), \\
  & X_{k,i}^{N,\eta}(0) = \xi_i^k, \quad i=1,\ldots,n,\ k=1,\ldots,N,
\end{aligned}
\end{equation}
where $f_\eta$ is a globally Lipschitz continuous approximation of $f$ with a
Lipschitz constant smaller or equal than $\eta^{-\alpha}$ for some small $\alpha>0$. 
In view of \eqref{1.scaling.B}, we can interpret the scaling parameter $\eta$ as the 
interaction radius of each particle. 

Equations \eqref{1.skt} are derived from system \eqref{1.pl} in the 
limit $N\to\infty$, $\eta\to 0$, with the scaling relation between $\eta$ 
and $N$ given in \eqref{1.scaling.deriv} below.
First, for fixed $\eta>0$, we perform a classical mean-field limit from 
\eqref{1.pl} to the following auxiliary {\rm intermediate} system:
\begin{equation}\label{1.il}
\begin{aligned}
  & \d \overline{X}_{k,i}^{\eta} = -\na U_i(\overline{X}_{k,i}^{\eta})\d t
	+ \bigg(2\sigma_i + 2\sum_{j=1}^nf_\eta\big(B_{ij}^\eta*
	u_{\eta,j}(\overline{X}_{k,i}^\eta)\big)\bigg)^{1/2}\d W_i^k(t), \\
	& \overline{X}_{k,i}^\eta(0) = \xi_i^k, \quad i=1,\ldots,n,\ k=1,\ldots,N,
\end{aligned}
\end{equation}
where we set $u_{\eta,j}(\overline{X}_{k,i}^\eta)=u_{\eta,j}
(t,\overline{X}_{k,i}^\eta(t))$ for $j=1,\ldots,n$. The function $u_{\eta,j}$ 
satisfies the nonlocal cross-diffusion system
\begin{equation}\label{1.nonloc}
\begin{aligned}
  & \pa_t u_{\eta,i} = \diver(u_{\eta,i}\na U_i) 
	+ \Delta\bigg(\sigma_i u_{\eta,i} + u_{\eta,i}\sum_{j=1}^n 
	f_\eta(B_{ij}^\eta * u_{\eta,j})\bigg), \\
	&	u_{\eta,i}(0) = u_i^0\ \mbox{in }\R^d,\quad i=1,\ldots,n,
\end{aligned}
\end{equation}
and will be later identified as the probability density function of 
$\overline{X}_{k,i}^{\eta}$. Note that we consider $N$ independent copies 
$\overline{X}_{k,i}^\eta$, $k=1,\ldots,N$, and the intermediate system depends on
$k$ only through the initial datum.

Then, passing to the limit $N\to\infty$, $\eta\to 0$ in \eqref{1.pl} leads to the 
{\em macroscopic system} 
\begin{equation}\label{1.ml}
\begin{aligned}
  & \d \widehat{X}_{k,i} = -\na U_i(\widehat{X}_{k,i})\d t
	+ \bigg(2\sigma_i + 2\sum_{j=1}^n f(a_{ij}u_{j}(\widehat{X}_{k,i}))\bigg)^{1/2}
	\d W_i^k(t), \\
	& \widehat{X}_{k,i}^\eta(0) = \xi_i^k, \quad i=1,\ldots,n,\ k=1,\ldots,N,
\end{aligned}
\end{equation}
where the functions $u_i$ satisfy \eqref{1.skt} and can be identified as 
the probability density functions of $\widehat{X}_{k,i}$.
In this limit, we assume that there exists $\delta>0$, depending on $n$, 
$\min_i\sigma_i$, and $T$, such that
\begin{equation}\label{1.scaling.deriv}
  \eta^{-2(d+1+\alpha)}\le\delta\log N
\end{equation}
holds, where $\alpha\geq 0$ depends on the Lipschitz condition of $f$, see 
Assumption $(A4)$ below, and that the function $f$ and its dervatives or, 
alternatively the initial
data, are sufficiently small (see Section \ref{sec.main} for details).
The main result of the paper is the error estimate
\begin{equation}\label{1.error}
  \sup_{k=1,\ldots,N}\E\bigg(\sum_{i=1}^n\sup_{0<s<T}\big|X_{k,i}^{N,\eta}(s)
	- \widehat{X}_{k,i}(s)\big|^2\bigg)\le C(T)\eta^{2(1-\alpha)}.
\end{equation}
We prove this estimate for the
potential $U_i(x)=-\frac12|x|^2$, but more general functions are possible;
see Remark \ref{rem.disc}. Note that estimate \eqref{1.error} implies 
propagation of chaos; see Remark \ref{prop.cha}.
In the case $\alpha=0$, our scaling \eqref{1.scaling.deriv} for the multi-species
case recovers the result in \cite{JoMe98}, where a single-species, moderately
interacting particle system with interaction in the diffusion part was considered.
\red{Our strategy is similar to that one of \cite{JoMe98} (and based on ideas
of Oelschl\"ager \cite{Oel89}). Since we allow for locally Lipschitz continuous
nonlinearities only, we obtain a smaller convergence rate compared to \cite{JoMe98},
which in fact is natural, since we approximate the nonlinearity
with functions having a Lipschitz constant of order
$\eta^{-\alpha}$. A difference to \cite{JoMe98} is that the authors
assume that the diffusion matrix in the stochastic part is positive definite. 
We do not suppose such a condition, but we need a smallness condition on the
nonlinearity for the existence proofs of systems \eqref{1.skt} and \eqref{1.nonloc}.}

Next, we present a brief overview on the existing literature concerning 
mean-field limits and moderately interacting many-particle limits in the context
of diffusion equations. Mean-field limits from stochastic differential equations have
been investigated since the 1980s; see the reviews \cite{Gol03,JaWa17} 
and the classical works by Sznitman \cite{Szn84,Szn91}. Oelschl\"ager
proved that in the many-particle limit, weakly interacting stochastic particle systems
converge to a deterministic nonlinear process \cite{Oel84}. Later, he generalized
his approach for systems of reaction-diffusion equations \cite{Oel89}
and porous-medium-type equations with quadratic diffusion \cite{Oel90},
by using moderately interacting particle systems. We also refer to the 
recent work \cite{CGK18}, which also includes numerical simulations.
As already mentioned, moderate interactions in stochastic particle system 
with nonlinear diffusion coefficients were investigated for the first time in 
\cite{JoMe98}. Later, Stevens derived the chemotaxis model from a 
many-particle system \cite{Ste00}. 
Further works concern the mean-field limit
leading to reaction-diffusion equations with nonlocal terms \cite{IRS12},
the hydrodynamic limit in a two-component system of Brownian motions to the
cross-diffusion Maxwell--Stefan equations \cite{Seo18}, 
and the large population limit of
point measure-valued Markov processes to nonlocal Lotka--Volterra systems with
cross diffusion \cite{FoMe15}. The latter model is similar to the nonlocal
system \eqref{1.nonloc}. The limit from the nonlocal to the local 
diffusion system was shown in \cite{Mou20} but only for triangular diffusion matrices.
The many-particle limit from a particle system driven by L\'evy noise 
to a fractional cross-diffusion system related to \eqref{1.cdj}
was recently shown in \cite{DPR20}. Furthermore, the population system \eqref{1.skt}
was derived in \cite{DDD19} from a time-continuous Markov chain model using the
BBGKY hierarchy. This paper presents, up to our knowledge, the first rigorous 
derivation of the Shigesada--Kawasaki--Teramoto (SKT) model \eqref{1.skt} from a 
stochastic particle system in the moderate many-particle limit.

Porous-medium-type equations can be derived from stochastic interacting particle
systems by assuming interactions in the drift term \cite{FiPh08}
or in the diffusion term \cite{JoMe98}. We allow for interactions in the
diffusion part but in a multi-species setting. The paper \cite{FoMe15} is concerned
with a multi-species framework too, but the authors assume bounded Lipschitz
continuous interaction potentials and derive a nonlocal cross-diffusion system only.
We are able to relax the assumptions and derive the local cross-diffusion system
\eqref{1.skt}.

Compared to the work \cite{DDD19}, we take the limits $N\to\infty$, $\eta\to 0$
simultaneously. However, our approach also implies the two-step limit.
Indeed, we can first perform the limit $N\to\infty$ for fixed $\eta>0$ and
afterwards the limit $\eta\to 0$ on the PDE level; see Lemma \ref{lem.conv1}
and Theorem \ref{thm.loc}. The simultaneous limit $N\to\infty$, $\eta\to 0$,
satisfying the scaling relation \eqref{1.scaling.deriv}, gives a more complete
picture, since we can prove the convergence in expectation for the difference
of the solutions to the stochastic systems \eqref{1.pl} and \eqref{1.ml}.

Finally, we remark that the cross-diffusion models \eqref{1.skt} and \eqref{1.cdj}
have quite different structural properties; also see \cite{BCPS20, BPRS20}. 
First, system \eqref{1.cdj} has a 
formal gradient-flow structure for each species separately, while system \eqref{1.skt}
can be written, under the detailed-balance condition \cite{CDJ19}, only in a
vector-valued gradient-flow form. Second, the segregation behavior of both models 
is different, i.e., segregation is stronger for the solutions to \eqref{1.cdj}
than for model \eqref{1.skt}; see the numerical experiments in Section \ref{sec.num}.

The paper is organized as follows. We present our assumptions and main results 
in Section \ref{sec.main}. The existence of smooth solutions to the
cross-diffusion systems \eqref{1.skt} and \eqref{1.nonloc} and an error 
estimate for the difference of the corresponding solutions is proved in
Sections \ref{sec.nonloc} and \ref{sec.loc}, respectively. 
The proofs are based on Banach's fixed-point theorem and 
higher-order estimations. We present the full proof since the environmental 
potential $U_i(x)=-\frac12|x|^2$ is not square-integrable, which requires some care;
see the arguments following \eqref{2.I7}.
Section \ref{sec.link} is concerned with the
identification of the solutions to the local and
nonlocal cross-diffusion systems \eqref{1.skt} and \eqref{1.nonloc}, respectively, 
with the probability density functions associated to the particle systems
\eqref{1.ml} and \eqref{1.il}, respectively. 
Error estimate \eqref{1.error}, the main result of the paper,
is proved in Section \ref{sec.conv}. 
In Section \ref{sec.num}, we present Monte--Carlo simulations for an
Euler--Maruyama discretization of system \eqref{1.pl} and compare them to
the numerical results from the particle system associated to \eqref{1.cdj}.
In the appendix, we recall some inequalities used in the paper.


\section{Assumptions and main results}\label{sec.main}

We impose the following assumptions:

\begin{enumerate}[({A}1)]
\item \textbf{Data:} $\sigma_i\in(0,\infty)$ and $\xi_i^1,\ldots,\xi_i^N$ are independent and
identically distributed (iid) square-integrable random variables 
with the common density function 
$u_{0,i}$ for $i=1,\ldots,n$ on the probability space $(\Omega,\mathcal{F},P)$. 
\item \textbf{Environmental potential:} $U_i(x)=-\frac12|x|^2$, $i=1,\ldots,n$.
\item \textbf{Interaction potential:} $B_{ij}\in C_0^\infty(\R^d)$ satisfies
$\operatorname{supp}(B_{ij})\subset B_1(0)$, where $B_1(0)$ is the unit ball in $\R^d$
and $i,j=1,\ldots,n$. 
\item \textbf{Nonlinearity:} $f\in W_{\rm loc}^{s+1,\infty}(\R;[0,\infty))$
and $f_\eta\in W^{s+1,\infty}(\R,[0,\infty))$ is such that
$f_\eta=f$ on $[-a_\eta,a_\eta]$ and the Lipschitz constant of $f_\eta$ is 
less than or equal to $\eta^{-\alpha}$ for a fixed $\alpha\in[0,1)$. 
Here, $s>d/2+1$ and $a_\eta\to\infty$ as $\eta\to 0$. If $f$ is globally 
Lipschitz continuous, we set $\alpha = 0$ and $f_\eta = f$.
\end{enumerate}

\begin{remark}[Discussion]\rm\label{rem.disc}
\textit{Environmental potential:}
The sign of $U_i$ guarantees that the populations are dispersed since the
drift term becomes $-x\cdot\na u_i-u_i$.
We have taken a quadratic potential $U_i$ to simplify the presentation.
\red{``Dispersive'' potentials (i.e.\ potentials $U_i$ with $\Delta U_i\le 0$)
are needed in the analysis, since we cannot bound terms including $\Delta U_i$
if $\Delta U_i\ge 0$.}
It is possible to choose general (dispersive) 
potentials $U_i\in C^\infty(\R^d)$ such that
$\na U_i$ is globally Lipschitz continuous, 
$D^kU_i\in L^\infty(\R^d)$ for $k=2,\ldots,s+2$,
the Hessian $D^2U_i$ is negative semidefinite, $\Delta U_i<0$, and
$D^kU_i$ for $k=3,\ldots,s$ is sufficiently small in the $L^\infty(\R^d)$ norm.
Thus, we may choose $U_i(x)=-|x|^2+g(x)$ and $g$ is a smooth
perturbation. 

\textit{Nonlinearity:} Since $f$ is not assumed to be globally Lipschitz continuous, 
we need to approximate the nonlinearity. The condition on the Lipschitz constant 
of $f_\eta$ ensures that we have a control on the growth of the Lipschitz constant 
of $f_\eta$ in the limit $N\to\infty$ and $\eta\to 0$. This growth condition is 
needed in the proof of Lemma \ref{lem.conv1}; see \eqref{4.aux2} and thereafter. 
The condition $s>d/2+1$ ensures that the embedding $H^s(\R^d)\hookrightarrow
W^{1,\infty}(\R^d)$ is continuous, and this embedding is needed to obtain
solutions in $H^s(\R^d)$ and to derive the estimates.
\qed
\end{remark}

We introduce some notation. We set 
$$
  a_{ij} = \int_{\R^d}B_{ij}(|x|)\d x, \quad i,j=1,\ldots,n,
$$
$B_{ij}^\eta(x)=\eta^{-d}B_{ij}(|x|/\eta)$, $A_{ij}=\|B_{ij}\|_{L^1(\R^d)}
=\|B_{ij}^\eta\|_{L^1(\R^d)}$ and $A=\max_{i,j=1,\ldots,n}A_{ij}$.
Let $C_s>0$ be the constant of the continuous
embedding $H^s(\R^d)\hookrightarrow L^\infty(\R^d)$ and set
\begin{equation}\label{def.I}
  I = [-2AC_s\|u_0\|_{H^s(\R^d)},2AC_s\|u_0\|_{H^s(\R^d)}].
\end{equation}
Then, for small $\eta>0$ such that $a_\eta\ge 2AC_s\|u_0\|_{H^s(\R^d)}$, 
we have $f_\eta=f$ on $I$.

First, we ensure that the nonlocal and local cross-diffusion systems 
\eqref{1.nonloc} and \eqref{1.skt}, respectively, have global smooth solutions.

\begin{theorem}[Existence for the nonlocal system]\label{thm.nonloc}
Let Assumptions (A2) and (A4) hold, 
$u_0\in H^s(\R^d;\R^n)$ for $s>d/2+1$, and let $\eta>0$ be such that
$a_\eta\ge 2AC_s\|u_0\|_{H^s(\R^d)}$.
There exists $\eps>0$ depending on $u_0$ such that if $\|f\|_{C^{s+1}(I)}\le\eps$,
system \eqref{1.nonloc} possesses a unique solution 
$u_\eta=(u_{\eta,1},\ldots,u_{\eta,n})$ satisfying
\begin{align*}
  & u_{\eta,i}\in L^\infty(0,\infty;H^s(\R^d))\cap 
	L^2(0,\infty;H^{s+1}(\R^d)), \\
  & \|u_\eta\|_{L^\infty(0,T;H^s(\R^d))}^2 
	+ \sigma_*\|\na u_\eta\|_{L^2(0,\infty;H^{s}(\R^d))}^2 
	\le \|u_0\|_{H^s(\R^d)}^2,
\end{align*}
where $0<\sigma_*<\sigma_{\rm min}:=\min_{i=1,\ldots,n}\sigma_i$.
\end{theorem}

The dependence of $\eps$ on $u_0$ can be made more explicit. The proof shows that
we need to choose $0<\eps < C\sigma_{\rm min}^{1/2}\|u_0\|^{-s}_{H^s(\R^d)}$, 
where $C>0$ is independent of $u_0$ and $\sigma_i$. 
Thus, if $\|f\|_{C^{s+1}(I)}$ is finite, the global existence result is valid for
small initial data.

\begin{theorem}[Existence for the local system]\label{thm.loc}
Let $u_0$ and $\eta$ satisfy the assumptions of Theorem \ref{thm.nonloc}.
Then there exists $\eps>0$ depending on $u_0$ such that if 
$\|f\|_{C^{s+1}(I)}\le\eps$,
system \eqref{1.skt} possesses a unique solution $u=(u_1,\ldots,u_n)$ satisfying
\begin{align*}
  & u_i\in L^\infty(0,\infty;H^s(\R^d))\cap L^2(0,\infty;H^{s+1}(\R^d)), 
	\quad i=1,\ldots,n, \\
  & \|u\|_{L^\infty(0,\infty;H^s(\R^d))}^2 
	+ \sigma_*\|\na u\|_{L^2(0,\infty;H^s(\R^d))}^2
	\le \|u_0\|_{H^s(\R^d)}^2,
\end{align*}
where $0<\sigma_*<\sigma_{\rm min}$.
Moreover, with the solution $u_\eta$ from Theorem \ref{thm.nonloc}, 
it holds that for an arbitrary $T>0$,
$$
  \|u-u_\eta\|_{L^\infty(0,T;L^2(\R^d))} + \|\na(u-u_\eta)\|_{L^2(0,T;L^2(\R^d))}
	\le C(T)\eta.
$$
\end{theorem}

Next, we state an existence result for the stochastic particle systems 
\eqref{1.pl}, \eqref{1.il}, and \eqref{1.ml}. 

\begin{proposition}\label{prop.ex}
\red{Let Assumptions (A1)--(A4) hold and let $\eta>0$, $N\in\N$}. Then:

{\em (i)} There exist unique square-integrable adapted
stochastic processes with continuous paths, which are strong solutions to systems 
\eqref{1.pl}, \eqref{1.il}, and \eqref{1.ml}, respectively.

{\rm (ii)} For each $t>0$, the $(nNd)$-dimensional random variables 
$\overline{X}^\eta(t)$ and $\widehat{X}(t)$ possess density
functions $\overline{u}_\eta(t)^{\otimes N}$ and $\widehat{u}(t)^{\otimes N}$ with
respect to the Lebesgue measure on $\R^{nNd}$, respectively.
\end{proposition}

The proof follows from \cite{KaSh91} and \cite{Nua06}. Indeed, 
Theorem 2.9 in \cite[page 289]{KaSh91} shows that there exist continuous 
square-integrable stochastic processes, which are strong solutions to 
\eqref{1.pl}, \eqref{1.il}, and \eqref{1.ml},
respectively. Strong uniqueness is guaranteed by Theorem 2.5 in \cite[page 287]{KaSh91}.
We conclude from \cite[Theorem 2.3.1]{Nua06} that
$\overline{X}_\eta(t)$ and $\widehat{X}(t)$ are absolutely continuous with respect
to the Lebesgue measure and thus, they possess density functions
$\overline{u}_\eta(t,x)^{\otimes N}$ and $\widehat{u}(t,x)^{\otimes N}$, 
respectively.
We prove in Section \ref{sec.link} that the density functions $\overline{u}_\eta$ and
$\widehat{u}$ can be identified with $u_\eta$ and $u$, the solutions to
\eqref{1.nonloc} and \eqref{1.skt}, respectively.

The following theorem is our main result. 

\begin{theorem}\label{thm.conv}
Let $X_{k,i}^{N,\eta}$ and $\widehat{X}_{k,i}$ be the solutions to
\eqref{1.pl} and \eqref{1.ml}, respectively. Then there exist parameters $\delta>0$,
depending on $n$, $\sigma_{\rm min}$, and $T$, 
and $\eps>0$, depending on $u_0$, such that if
$\eta^{-2(d+1+\alpha)}\le\delta\log N$ and $\|f\|_{C^{s+1}(I)}\le\eps$, 
$$
  \sup_{k=1,\ldots,N}\E\bigg(\sum_{i=1}^n\sup_{0<s<T}
	\big|(X_{k,i}^{N,\eta}-\widehat{X}_{k,i})(s)\big|^2\bigg) 
	\le C(T,n,\sigma_{\rm min})\eta^{2(1-\alpha)},
$$
where $\alpha\ge 0$ is defined in Assumption (A4).
\end{theorem}
\begin{remark}\label{prop.cha}\rm
It is well-known that this result implies \textit{propagation of chaos} in the 
single-species case; see, e.g., \cite[Section 3.1]{JaWa17}. In the multi-species case, 
this generalizes for fixed $k$ to the convergence of the $k$-marginal distribution 
$F_k(t)$ of $(X_{{j_1},i_1}^{N,\eta}(t), \ldots, X_{{j_{k}},i_k}^{N,\eta}(t))$ 
at any time $t>0$ 
towards the product measure $\otimes_{\ell=1}^k u_{i_\ell}(\cdot,t)$ as 
$N\to\infty$, $\eta\to 0$, i.e.
$$
  W_2^2\bigg(F_k(t),\bigotimes_{\ell=1}^k u_{i_\ell}(\cdot,t)\bigg) 
	\leq kC(T,n,\sigma_{\rm min})\eta \to 0,
$$
where $W_2$ denotes the 
2-Wasserstein distance.
\qed
\end{remark}


\section{Proof of Theorem \ref{thm.nonloc}}\label{sec.nonloc}

We prove the global existence of smooth solutions to the nonlocal system
\eqref{1.nonloc}. Since $\eta$ is fixed in the proof, 
we omit it for $u_\eta$ to simplify the notation. We split the proof in several steps.
\red{In the first step, we prove the existence of local-in-time solutions
satisfying $\|u_i(t)\|_{H^s(\R^d)}\le 2\|u_0\|_{H^s(\R^d)}$ for $0<t<T(\eta)$ for some
(possibly) small $T(\eta)>0$. 
Actually, we show in the second step, that the factor 2 can be replaced by one.
This uniform estimate allows us in the third step to conclude the global existence.} 

{\em Step 1: Local existence of solutions.} In this step,
the smallness conditions on $\eta$ and $f$ are not needed.
The idea is to apply the Banach fixed-point theorem on the space
$$
  X_T := \big\{v\in L^\infty(0,T;H^s(\R^d;\R^n)):
	\|v\|_{L^\infty(0,T;H^s(\R^d))} \le 2\|u_0\|_{H^s(\R^d)}\big\},
$$
where $T>0$ will be determined later in this proof. 
We define the fixed-point operator $S:X_T\to X_T$, $S(v)=u$,
where $u$ is the unique solution to the linear problem
\begin{equation}\label{2.lin}
  \pa_t u_i = \diver(u_i\na U_i) + \Delta\big(u_i(\sigma_i + K_i(v(t,x)))\big),
	\quad u_i(0)=u_{0,i}\quad\mbox{in }\R^d,\ t>0,
\end{equation}
with $K_i(v)=\sum_{j=1}^n f_\eta(B_{ij}^\eta*v_j)\ge 0$, $i=1,\ldots,n$. 
We need to show that $S$ is well defined. We infer from Young's convolution inequality
(Lemma \ref{lem.convul}) 
and the embedding $H^s(\R^d)\hookrightarrow L^{\infty}(\R^d)$ that
\begin{align}
  \sup_{0<t<T}\|\na K_i(v)\|_{L^\infty(\R^d)}
  &\le \sum_{j=1}^n\|f'_\eta\|_{L^\infty(\R)}\|\na B_{ij}^\eta\|_{L^1(\R^d)}
	\sup_{0<t<T}\|v_j(t)\|_{L^\infty(\R^d)} \nonumber \\
	&\le C(\eta) \sum_{j=1}^n\|v_j\|_{L^\infty(0,T;H^s(\R^d))}
	< \infty, \label{2.Ki}
\end{align}
i.e., $K_i(v)$ is globally Lipschitz continuous. Therefore, a Galerkin argument
to verify higher-order regularity shows that, for given $v\in X_T$, 
there exists a unique solution 
$u_i\in L^\infty(0,T;H^{s}(\R^d))\cap L^2(0,T;H^{s+1}(\R^d))$ to
\eqref{2.lin}. It remains to show that $u=(u_1,\ldots,u_n)\in X_T$ for some $T>0$.
The estimations are not difficult, but since $\na U_i$ is not square
integrable, some care is needed.

First, we prove higher-order estimates for $K_i(v)$. 
Let $\alpha\in\N_0^d$ be a multi-index with order $|\alpha|=m\leq s$.
By Lemma \ref{lem.moser2} and Young's convolution inequality,
\begin{align}
  \int_0^T\|D^\alpha K_i(v)\|_{L^2(\R^d)}^2\d t 
	&\le C\int_0^T\sum_{j=1}^n\|f'_\eta\|_{C^{m-1}(\R)}^2
	\|B_{ij}^\eta*v_j\|_{L^\infty(\R^d)}^{2(m-1)}
	\|D^\alpha(B_{ij}^\eta*v_j)\|_{L^2(\R^d)}^2 \d t \nonumber \\
  &\le C(\eta)\int_0^T\sum_{j=1}^n\|B_{ij}^\eta\|_{L^1(\R^d)}^{2m}
	\|v_j\|_{L^\infty(\R^d)}^{2(m-1)}\|D^\alpha v_j\|_{L^2(\R^d)}^2 \d t 
	\nonumber \\
	&\le C(\eta)\sum_{j=1}^n\int_0^T\|v_j\|_{H^s(\R^d)}^{2m} \d t < \infty, 
	\label{2.Ki2}
\end{align}
where here and in the following, $C>0$, $C(\eta)>0$, etc.\ are generic constants 
with values changing from line to line.
In a similar way, applying Lemmas \ref{lem.convul} and \ref{lem.moser},
\begin{align}
	\sup_{0<t<T}\|&D^\alpha\na K_i(v)\|_{L^2(\R^d)}^2
	\le C\sup_{\red{0<t<T}}\sum_{j=1}^n\big\|D^\alpha\big(f_\eta'(B_{ij}^\eta*v_j)
	\na B_{ij}^\eta*v_j\big)\big\|_{L^2(\R^d)}^2 \nonumber \\
	&\le C\sup_{0<t<T}\sum_{j=1}^n\Big(\|f'_\eta(B_{ij}^\eta*v_j)\|_{L^\infty(\R^d)}
	\|\na B_{ij}^\eta\|_{L^1(\R^d)}\|D^{m} v_j\|_{L^2(\R^d)} \nonumber \\
	&\phantom{xx}{}+ \|D^{m} (f'_\eta(B_{ij}^\eta*v_j))\|_{L^2(\R^d)}
	\|\na B_{ij}^\eta\|_{L^1(\R^d)}\|v_j\|_{L^\infty(\R^d)}\Big)^2
	\le C(\eta), \label{2.Ki3}
\end{align}
since, according to Lemma \ref{lem.moser2},
we can bound $\sup_{0<t<T}\|D^{m}(f'_\eta(B_{ij}*v_j))\|_{L^2(\R^d)}$ 
in terms of $\|f_\eta\|_{C^{s+1}(\R)}$, $\|B_{ij}^\eta\|_{L^1(\R^d)}$,
and $\sup_{0<t<T}\|v_j\|_{H^s(\R^d)}$, and it holds that 
$\|\na B_{ij}^\eta\|_{L^1(\R^d)}\le C(\eta)$.

We proceed with the proof of $u\in X_T$ for some $T>0$.
Applying $D^\alpha$ to \eqref{2.lin}, multiplying the resulting equation by
$D^\alpha u_i$, and integrating over $(0,\tau)\times\R^d$ for $\tau<T$ yields
\begin{equation}\label{2.aux}
  \frac12\int_{\R^d} |D^\alpha u_i(\tau)|^2 \d x 
	- \frac12\int_{\R^d}|D^\alpha u_{0,i}|^2\d x
	+ \sigma_i\int_0^\tau\int_{\R^d}|\na D^\alpha u_i|^2 \d x\d t
	= I_1 + I_2 + I_3,
\end{equation}
where
\begin{align*}
  I_1 &= -\int_0^\tau\int_{\R^d}\na D^\alpha u_i\cdot D^\alpha(u_i\na U_i)\d x\d t, \\
	I_2 &= -\int_0^\tau\int_{\R^d}\na D^\alpha u_i\cdot 
	D^\alpha(\na u_iK_i(v))\d x\d t, \\
  I_3 &= -\int_0^\tau\int_{\R^d}\na D^\alpha u_i\cdot 
	D^\alpha(u_i\na K_i(v))\d x\d t.
\end{align*}

First, let $|\alpha|=m=0$. Then, integrating by parts in $I_1$,
using Young's inequality, and observing that $U_i (x)= -\frac12|x|^2$,
\begin{align*}
  I_1 &= \frac12\int_0^\tau\int_{\R^d}u_i^2\Delta U_i \d x\d t
	= -\frac{d}{2}\int_0^\tau\int_{\R^d}u_i^2\d x\d t \le 0, \\
	I_2 &= -\int_0^\tau\int_{\R^d}K_i(v)|\na u_i|^2 \d x\d t \le 0, \\
  I_3 &\le \frac{\sigma_i}{2}\int_0^\tau\int_{\R^d}|\na u_i|^2\d x\d t 
	+ \frac{1}{2\sigma_i}
	\|\na K_i(v)\|_{L^\infty(0,T;L^\infty(\R^d))}^2\int_0^\tau \|u_i\|_{L^2(\R^d)}^2\d t,
\end{align*}
where we used $K_i(v)\ge 0$ for $I_2$. 
It follows from \eqref{2.Ki} that
$$
  I_1+I_2+I_3 \le \frac{\sigma_i}{2}\int_0^\tau\int_{\R^d}|\na u_i|^2\d x\d t
	+ C\int_0^\tau\|u_i\|_{L^2(\R^d)}^2\d t,
$$
where $C>0$ depends on the $L^\infty(0,T;H^s(\R^d))$ norm of $v$.
Inserting this estimate into \eqref{2.aux} with $\alpha=0$ and 
applying the Gronwall inequality, we infer that
$$
  \int_{\R^d}u_i(\tau)^2\d x + \frac{\sigma_i}{2}
	\int_0^\tau\int_{\R^d}|\na u_i|^2\d x\d t \le C(u_0)e^{C\tau}.
$$
This shows that $u_i$ is bounded in $L^\infty(0,T;L^2(\R^d))$ and
$L^2(0,T;H^1(\R^d))$. 

Now, let $|\alpha|=m\ge 1$. Then, integrating by parts, using $\Delta U_i\le 0$, 
and applying Young's inequality again,
\begin{align*}
  I_1 &= \frac12\int_0^\tau\int_{\R^d}(D^\alpha u_i)^2\Delta U_i \d x\d t
	- \int_0^\tau\int_{\R^d}\na D^\alpha u_i\cdot\big(D^\alpha(u_i\na U_i)
	-D^\alpha u_i\na U_i\big)\d x\d t \\
	&\le \frac{\sigma_i}{4}\int_0^\tau\int_{\R^d}|\na D^\alpha u_i|^2 \d x\d t
	+ \sum_{0<|\beta|\le|\alpha|}\int_0^\tau c_\beta
  \|D^{\alpha-\beta}u_i\|_{L^2(\R^d)}^2 \|D^\beta\na U_i\|_{L^\infty(\R^d)}^2 \d t \\
	&\le \frac{\sigma_i}{4}\int_0^\tau\int_{\R^d}|\na D^\alpha u_i|^2 \d x\d t
	+ C\int_0^\tau\|u_i\|_{H^{m-1}(\R^d)}^2\d t,
\end{align*}
where we used the fact that $D^\beta\na U_i$ is bounded for $|\beta|=1$ and vanishes
for $|\beta|>1$. It follows from integration by parts, $K_i(v)\ge 0$, and
Lemma \ref{lem.comm} that
\begin{align*}
  I_2 &= -\int_0^\tau\int_{\R^d}\na D^\alpha u_i\cdot\big(D^\alpha(\na u_i K_i(v))
	- \na D^\alpha u_i K_i(v)\big)\d x\d t \\
	&\phantom{xx}{}- \int_0^\tau\int_{\R^d}K_i(v)|\na D^\alpha u_i|^2\d x\d t \\
	&\le \frac{\sigma_i}{4}\int_0^\tau\int_{\R^d}|\na D^\alpha u_i|^2\d x\d t 
	+ C\int_0^\tau\big(\|DK_i(v)\|_{L^\infty(\R^d)}
	\|D^{m-1}\na u_i\|_{L^2(\R^d)} \\
	&\phantom{xx}{}
	+ \|D^{m} K_i(v)\|_{L^2(\R^d)}\|\na u_i\|_{L^\infty(\R^d)}\big)^2\d x\d t.
\end{align*}
We infer from estimates \eqref{2.Ki} and \eqref{2.Ki2} for $K_i(v)$ and the
embedding $H^s(\R^d)\hookrightarrow W^{1,\infty}(\R^d)$ that
$$
  I_2 \le \frac{\sigma_i}{4}\int_0^\tau\int_{\R^d}|\na D^\alpha u_i|^2\d x\d t
	+ C\int_0^\tau\|u_i\|_{H^s(\R^d)}^2\d t.
$$
Finally, we use Lemma \ref{lem.moser} and estimates \eqref{2.Ki} and \eqref{2.Ki3}
to obtain
\begin{align*}
  I_3 &\le \frac{\sigma_i}{4}\int_0^\tau\int_{\R^d}|\na D^\alpha u_i|^2 \d x\d t 
	+ C\int_0^\tau\int_{\R^d}\big(\|u_i\|_{L^\infty(\R^d)}
	\|D^{m}\na K_i(v)\|_{L^2(\R^d)} \\
	&\phantom{xx}{}
	+ \|D^{m} u_i\|_{L^2(\R^d)}\|\na K_i(v)\|_{L^\infty(\R^d)}\big)^2\d x\d t \\
  &\le \frac{\sigma_i}{4}\int_0^\tau\int_{\R^d}|\na D^\alpha u_i|^2 \d x\d t
	+ C(\eta)\int_0^\tau\|u_i\|_{H^s(\R^d)}^2\d t.
\end{align*}

Inserting these estimates into \eqref{2.aux} and summing over $|\alpha|\le s$,
we arrive at
$$
  \|u_i(\tau)\|_{H^s(\R^d)}^2 
	+ \frac{\sigma_i}{4}\int_0^\tau\|\na u_i\|_{H^s(\R^d)}^2\d t
  \le \|u_{0,i}\|_{H^s(\R^d)}^2 + C(\eta)\int_0^\tau\|u_i\|_{H^s(\R^d)}^2\d t.
$$
Summing over $i=1,\ldots,n$ and applying Gronwall's inequality gives
$$
  \|u(\tau)\|_{H^s(\R^d)}^2\le \|u_0\|_{H^s(\R^d)}^2 e^{C(\eta)\tau}
	\le \|u_0\|_{H^s(\R^d)}^2 e^{C(\eta)T}.
$$
Choosing $T>0$ sufficiently small, we can ensure that $\|u(\tau)\|_{H^s(\R^d)}
\le 2\|u_0\|_{H^s(\R^d)}$ for all $0<\tau<T$. This shows that $u\in X_T$, i.e.,
the operator is well-defined.

Next, we prove that $S:X_T\to X_T$ is a contraction. Let $v$, $w\in X_T$ and set
$\bar v=S(v)$ and $\bar w=S(w)$. Taking the difference of equations \eqref{2.lin}
satisfied by $\bar v_i$ and $\bar w_i$, respectively, using the test function
$\bar v_i-\bar w_i$, and integrating by parts, it follows that
\begin{equation}\label{2.aux2}
  \frac12\int_{\R^d}(\bar v_i-\bar w_i)(\tau)^2\d x
	+ \sigma_i\int_0^\tau\int_{\R^d}|\na(\bar v_i-\bar w_i)|^2\d x\d t 
	= I_4+I_5+I_6,
\end{equation}
where
\begin{align*}
  I_4 &= \frac12\int_0^\tau\int_{\R^d}\Delta U_i(\bar v_i-\bar w_i)^2\d x\d t
	\le 0, \\
  I_5 &= -\int_0^\tau\int_{\R^d}\na\big((\bar v_i-\bar w_i)K_i(v)\big)\cdot
	\na(\bar v_i-\bar w_i)\d x\d t, \\
	I_6 &= -\int_0^\tau\int_{\R^d}\na\big(\bar w_i(K_i(v)-K_i(w))\big)\cdot\na
	(\bar v_i-\bar w_i)\d x\d t.
\end{align*}
Because of $K_i(v)\ge 0$ and estimate \eqref{2.Ki} for $\na K_i(v)$, we find that,
by Young's inequality,
\begin{align*}
  I_5 &= -\int_0^\tau\int_{\R^d}K_i(v)|\na(\bar v_i-\bar w_i)|^2\d x\d t
	- \int_0^\tau\int_{\R^d}(\bar v_i-\bar w_i)\na K_i(v)\cdot\na(\bar v_i-\bar w_i)
	\d x\d t \\
	&\le \frac{\sigma_i}{4}\int_0^\tau\|\na(\bar v_i-\bar w_i)\|_{L^2(\R^d)}^2\d t
	+ C(\sigma_i)\int_0^\tau\|\bar v_i-\bar w_i\|_{L^2(\R^d)}^2
	\|\na K_i(v)\|_{L^\infty(\R^d)}^2\d t \\
	&\le \frac{\sigma_i}{4}\int_0^\tau\|\na(\bar v_i-\bar w_i)\|_{L^2(\R^d)}^2\d t
	+ C(\eta)\int_0^\tau\|\bar v_i-\bar w_i\|_{L^2(\R^d)}^2\d t.
\end{align*}
It follows again from Young's inequality that
\begin{align}
  I_6 &\le \frac{\sigma_i}{4}\int_0^\tau\|\na(\bar v_i-\bar w_i)\|_{L^2(\R^d)}^2\d t
	+ C(\sigma_i)\int_0^\tau\|\na \bar w_i\|_{L^\infty(\R^d)}^2
	\|K_i(v)-K_i(w)\|_{L^2(\R^d)}^2\d t \nonumber \\
	&\phantom{xx}{}+ C(\sigma_i)\int_0^\tau\|\bar w_i\|_{L^\infty(\R^d)}^2
	\|\na(K_i(v)-K_i(w))\|_{L^2(\R^d)}^2\d t. \label{2.aux4}
\end{align}
Since $\bar w\in X_T$, we have $\|\na\bar w_i\|_{L^\infty(\R^d)}\le 
C\|\bar w_i\|_{H^s(\R^d)}\le C(u_0)$ and $\|\bar w_i\|_{L^\infty(\R^d)}\le C(u_0)$.
We use the fact that $f_\eta$ and $f'_\eta$ are globally Lipschitz continuous:
\begin{align*}
  \|K_i(v)-K_i(w)\|_{L^2(\R^d)}
	&\le C(\eta)\sum_{j=1}^n\|B_{ij}^\eta*(v_j-w_j)\|_{L^2(\R^d)} 
	\le  C(\eta)\|v-w\|_{L^2(\R^d)}, \\
	\|\na(K_i(v)-K_i(w))\|_{L^2(\R^d)}
	&\le \sum_{j=1}^n\|(f'_\eta(B_{ij}^\eta*v_j)-f'_\eta(B_{ij}^\eta*w_j))
	B_{ij}^\eta*\na v_j\|_{L^2(\R^d)} \\
	&\phantom{xx}{}
	+ \sum_{j=1}^n\|f'_\eta(B_{ij}^\eta*w_j)\na B_{ij}^\eta*(v_j-w_j)\|_{L^2(\R^d)} \\
	&\le C(\eta)\sum_{j=1}^n\|v_j-w_j\|_{L^2(\R^d)}\|B_{ij}^\eta\|_{L^1(\R^d)}
	\|\na v_j\|_{L^\infty(\R^d)} \\
	&\phantom{xx}{}+ C(\eta)\sum_{j=1}^n\|\na B_{ij}^\eta\|_{L^1(\R^d)}
	\|v_j-w_j\|_{L^2(\R^d)} \\
	&\le C(\eta)\|v-w\|_{L^2(\R^d)}.
\end{align*}
Inserting these inequalities into \eqref{2.aux4} and summarizing the estimates
for $I_4$, $I_5$, and $I_6$, we conclude from \eqref{2.aux2} and
summation over $i=1,\ldots,n$ that
\begin{align*}
  \frac12\|&(\bar v-\bar w)(\tau)\|_{L^2(\R^d)}^2
	+ \sum_{i=1}^n\frac{\sigma_i}{4}\int_0^\tau
	\|\na(\bar v_i-\bar w_i)\|_{L^2(\R^d)}^2\d t \\
	&\le C_1\int_0^\tau\|\bar v-\bar w\|_{L^2(\R^d)}^2\d t
  + C_2\tau\|v-w\|_{L^\infty(0,\tau;L^2(\R^d))}^2.
\end{align*}
We apply Gronwall's inequality and the supremum over $0<\tau<T$ to find that
$$
  \|\bar v-\bar w\|_{L^\infty(0,T;L^2(\R^d))}^2
	\le \red{2}C_2 e^{\red{2}C_1T}T\|v-w\|_{L^\infty(0,T;L^2(\R^d))}^2.
$$
Thus, choosing $T>0$ such that $\red{2}C_2e^{\red{2}C_1T}T<1$, 
we infer that $S:X_T\to X_T$
is a contraction. By Banach's fixed-point theorem, there exists a unique solution
$u\in L^\infty(0,T;H^s(\R^d))\cap L^2(0,T;H^{s+1}(\R^d))$ 
to \eqref{1.nonloc}.

{\em Step 2: \red{A priori} estimates.} Let $u=u_\eta$ be the unique solution to
\eqref{1.nonloc}. We know from Step 1 that
$\|u_i(t)\|_{L^\infty(\R^d)}\le \red{C_s}\|u_i(t)\|_{H^s(\R^d)}
\le 2C_s\|u_0\|_{H^s(\R^d)}$
for any $0<t<T$. Recall that $T=T(\eta)$ and hence we do not have uniform estimates 
in $\eta$ even for small $T>0$ at this step. 
\red{We show in this step the estimate $\|u_i(t)\|_{H^s(\R^d)} 
\le \|u_0\|_{H^s(\R^d)}$, which allows us to conclude that the
end time $T$ can be arbitrary and actually does not depend on $\eta$.}
We apply $D^\alpha$ to \eqref{1.nonloc} (with $|\alpha|=m\leq s$), multiply
the resulting equation by $D^\alpha u_i$, and integrate over $(0,\tau)\times\R^d$ for
$\tau<T$, similarly to the corresponding estimate in Step 1:
\begin{equation}\label{2.aux3}
  \frac12\int_{\R^d}|D^\alpha u_i(\tau)|^2\red{\d x}
	- \frac12\int_{\R^d}|D^\alpha u_{0,i}|^2\red{\d x}
	+ \sigma_i \int_0^\tau\int_{\R^d}|\na D^\alpha u_i|^2\d x\d t
	= I_7 + I_8 + I_9,
\end{equation}
where
\begin{align*}
  I_7 &= -\int_0^\tau\int_{\R^d}\na D^\alpha u_i\cdot D^\alpha(u_i\na U_i)\d x\d t, \\
	I_8 &= -\int_0^\tau\int_{\R^d}\na D^\alpha u_i\cdot D^\alpha(\na u_iK_i(u))\d x\d t,
	\\
	I_9 &= -\int_0^\tau\int_{\R^d}\na D^\alpha u_i\cdot D^\alpha(u_i \na K_i(u))\d x\d t,
\end{align*}
and we recall that $K_i(u)=\sum_{j=1}^n f_\eta(B_{ij}^\eta*u_j)$.

First, let $m=0$. Arguing similarly as for $I_1$ and $I_2$, we find that $I_7\le 0$
and $I_8\le 0$.
We estimate $\na K_i(u)=\sum_{j=1}^n f'_\eta(B_{ij}^\eta*u_j)B_{ij}^\eta*\na u_j$:
\begin{equation}\label{2.Ku}
  \|\na K_i(u)\|_{L^2(\R^d)}
	\le A\sum_{j=1}^n\|f'_\eta(B_{ij}^\eta*u_j)\|_{L^\infty(\R^d)}
	\|\na u_j\|_{L^2(\R^d)},
\end{equation}
recalling that $A=\max_{i,j=1,\ldots,n}$ $\|B_{ij}^\eta\|_{L^1(\R^d)}$. This gives 
for $m=0$:
\begin{align*}
	I_9 &\le \|u_i\|_{L^\infty(0,\tau;L^\infty(\R^d))}\int_0^\tau
  \|\na u_i\|_{L^2(\R^d)}\|\na K_i(u)\|_{L^2(\R^d)}\d t \\
	&\le C\|u_0\|_{H^s(\R^d)}\sum_{j=1}^n
	\|f'_\eta(B_{ij}^\eta*u_j)\|_{L^\infty(0,\tau;L^\infty(\R^d))}
	\int_0^\tau\|\na u_j\|_{L^2(\R^d)}^2\d t.
\end{align*}

From this point on, we will need the smallness condition on $f_\eta$ and $f'_\eta$.
Because of 
\begin{equation}\label{2.Bu}
  \|B_{ij}^\eta*u_j(t)\|_{L^\infty(\R^d)}\le \|B_{ij}^\eta\|_{L^1(\R^d)}
	C_s\|u_j(t)\|_{H^s(\R^d)} \le 2AC_s\|u_0\|_{H^s(\R^d)},
\end{equation}
where $C_s>0$ is the constant of the embedding $H^s(\R^d)\hookrightarrow 
L^\infty(\R^d)$, $(B_{ij}^\eta*u_j(t))(x)$ lies in the interval
$I=[-2AC_s\|u_0\|_{H^s(\R^d)},2AC_s\|u_0\|_{H^s(\R^d)}]$ for $0<t<T$ and $x\in\R^d$.
On this interval, $f_\eta=f$ if $\eta>0$ is sufficiently small. 
From now on, we use $f\le \eps$ and $|f'|\le\eps$ on $I$ for a small 
$\eps>0$. Thus, we have
$$
	I_9 \le C\eps\|u_0\|_{H^s(\R^d)}\red{\sum_{j=1}^n
	\int_0^\tau\|\na u_j\|_{L^2(\R^d)}^2\d t}.
$$

Inserting these estimates into \eqref{2.aux3}, we conclude that
$$
  \|u_i(\tau)\|_{L^2(\R^d)}^2 + \big(\sigma_i-C\eps\|u_0\|_{H^s(\R^d)}\big)
	\int_0^\tau\|\na u_i\|_{L^2(\R^d)}^2\d t\le \|u_{0,i}\|_{L^2(\R^d)}^2.
$$
Choosing $\eps>0$ sufficiently small, this gives an estimate for $u_i$
in $L^\infty(0,T;L^2(\R^d))\cap L^2(0,T;H^1(\R^d))$.

Next, let $m\ge 1$. The estimate for $I_7$ is delicate since 
$\na U_i\not\in L^2(\R^d)$, and the corresponding estimate for $I_1$ cannot
be directly used. We split $I_7$ into two parts:
\begin{align}
  I_7 &= \int_0^\tau\int_{\R^d}D^\alpha u_i D^\alpha(\na u_i\cdot\na U_i 
	+ u_i\Delta U_i)\d x\d t \nonumber \\
	&= \int_0^\tau\int_{\R^d}D^\alpha u_i\big(D^\alpha(\na u_i\cdot\na U_i)
	- D^\alpha\na u_i\cdot\na U_i\big)\d x\d t \nonumber \\
  &\phantom{xx}{}+ \int_0^\tau\int_{\R^d}D^\alpha u_i\big(D^\alpha(u_i\Delta U_i) 
	- D^\alpha u_i\Delta U_i\big)\d x\d t, \label{2.I7}
\end{align}
noting that the second terms in both integrals are the same (with different signs)
because of
$$
  -\int_{\R^d}D^\alpha u_i D^\alpha\na u_i\cdot\na U_i\d x
	= -\frac12\int_{\R^d}\na(D^\alpha u_i)^2\cdot\red{\na} U_i\d x
	= \frac12\int_{\R^d}(D^\alpha u_i)^2\Delta U_i\d x.
$$
Moreover, the last integral in \eqref{2.I7} vanishes since $\Delta U_i=-\red{d}$. 
In the first integral of the right-hand side of \eqref{2.I7}, the first-order
derivative of $U_i$ cancels, while the second-order derivative equals
$\pa^2 U_i/\pa x_j\pa x_k=-\delta_{jk}$ and all higher-order derivatives of $U_i$
vanish. Then a straightforward computation leads to
$$
  I_7 = -d\int_0^\tau\int_{\R^d}(D^\alpha u_i)^2 \d x\d t \le 0.
$$

For the estimates of $I_8$ and $I_9$, we need
a smallness condition on $f$ and its derivatives. 
We apply Young's inequality and Lemma \ref{lem.moser} to estimate 
the (more delicate) term $I_9$:
\begin{align*}
  I_9 &\le \frac{\sigma_i}{4}\int_0^\tau\|\na D^\alpha u_i\|_{L^2(\R^d)}^2\d t
	+ C(\sigma_i)\int_0^\tau\|D^\alpha(u_i\na K_i(u))\|_{L^2(\R^d)}^2\d t \\
	&\le \frac{\sigma_i}{4}\int_0^\tau\|\na D^\alpha u_i\|_{L^2(\R^d)}^2\d t 
	+ C\int_0^\tau\big(\|u_i\|_{L^\infty(\R^d)}
	\|D^{m}\na K_i(u)\|_{L^2(\R^d)} \\
	&\phantom{xx}{}
	+ \|D^{m} u_i\|_{L^2(\R^d)}\|\na K_i(u)\|_{L^\infty(\R^d)}\big)^2\d t.
\end{align*}
Estimate \eqref{2.Bu} shows that $f_\eta=f$ and $|f'|\le\eps$ on $I$. Then,
by similar arguments leading to \eqref{2.Ku}, 
\begin{align*}
  \|\na K_i(u)\|_{L^\infty(\R^d)}
	&\red{\le A\sum_{j=1}^n\|f_\eta'(B_{ij}^\eta*u_j)\|_{L^\infty(\R^d)}
	\|\na u_j\|_{L^\infty(\R^d)}} \\
	&\le A\eps\|\na u\|_{L^\infty(\R^d)}
	\le \eps AC_s\|\na u\|_{H^s(\R^d)}.
\end{align*}
Moreover, using Lemma \ref{lem.moser2}, the embedding $H^s(\R^d)\hookrightarrow
W^{1,\infty}(\R^d)$, and $m\leq s$,
\begin{align*}
  \|D^{m}&\na K_i(u)\|_{L^2(\R^d)} 
	\le A\sum_{j=1}^n\|\na u_j\|_{L^\infty(\R^d)}
	\|D^{m}(f'_\eta(B_{ij}^\eta*u_j))\|_{L^2(\R^d)} \\
	&\le C\sum_{j=1}^n\|\na u_j\|_{H^s(\R^d)}\|f''\|_{C^{m-1}(I)}
	\|B_{ij}^\eta*u_j\|_{L^\infty(\R^d)}^{m-1}
	\|B_{ij}^\eta*D^{m} u_j\|_{L^2(\R^d)} \\
	&\le \eps C\|\na u\|_{H^s(\R^d)}\|u\|_{L^\infty(\R^d)}^{m-1}
	\|D^{m} u\|_{L^2(\R^d)}
	\le \eps C\|\na u\|_{H^s(\R^d)}\|u_0\|_{H^s(\R^d)}^{s},
\end{align*}
recalling definition \eqref{def.I} of the interval $I$.
Consequently, the estimate for $I_9$ becomes
$$
  I_9 \le \frac{\sigma_i}{4}\int_0^\tau\|\na D^\alpha u_i\|_{L^2(\R^d)}^2\d t
	+ C\eps^2\|u_0\|_{H^s(\R^d)}^{2s}\int_0^\tau\|\na u\|_{H^s(\R^d)}^2\d t.
$$
The term $I_8$ is treated in a similar way, resulting in
$$
  I_8 \le \frac{\sigma_i}{4}\int_0^\tau\|\na D^\alpha u_i\|_{L^2(\R^d)}^2\d t
	+ C\eps^2\|u_0\|_{H^s(\R^d)}^{2s}\int_0^\tau\|\na u\|_{H^s(\R^d)}^2\d t.
$$
Set $\sigma_{\rm min}=\min_{i=1,\ldots,n}\sigma_i>0$.
We conclude from \eqref{2.aux3} after summation over $|\alpha|\le s$ and
$i=1,\ldots,n$ that
$$
  \|u(\tau)\|_{H^s(\R^d)}^2 + \big(\sigma_{\rm min}-C\eps^2\|u_0\|_{H^s(\R^d)}^s\big)
	\int_0^\tau\|\na u\|_{H^s(\R^d)}^2\d t 
	\le \|u_0\|_{H^s(\R^d)}^2.
$$
Thus, for sufficiently small $\eps>0$, we arrive at the desired estimate
uniform in $\eta$.

{\em Step 3: Global existence and uniqueness.} We have proved that
$\|u(\tau)\|_{H^s(\R^d)}\le\|u_0\|_{H^s(\R^d)}$ for $0<\tau\le T$ for some
sufficiently small $T>0$. The value for $T$ does not depend on the solution.
Thus, we can use $u(T)$ as an initial datum and solve the equation in
$[T,2T]$. Repeating this argument leads to a global solution. The uniqueness
of a solution follows after standard estimates, based on the global Lipschitz
continuity of $f_\eta$ and $f'_\eta$ (see the calculations for $I_4$, $I_5$, and
$I_6$) and choosing $\eps>0$ sufficiently small.


\section{Proof of Theorem \ref{thm.loc}}\label{sec.loc}

We show the global existence of smooth solutions to the local system \eqref{1.skt}
and an error estimate for the difference of the solutions to \eqref{1.skt} and
\eqref{1.nonloc}, respectively. 
\red{First, we prove that a solution $u_\eta$ to \eqref{1.nonloc} converges
to a solution $u$ to \eqref{1.skt} in a certain sense. Then we prove the
error bound in Theorem \ref{thm.loc} by estimating the difference $u_\eta-u$.
The key of the proof is the estimate of the difference
$f_\eta(B_{ij}^\eta*u_{\eta,j})-f_\eta(a_{ij}u_{\eta,j})$.}

{\em Step 1. Existence and uniqueness of solutions.}
Let $u_\eta$ be a smooth solution to \eqref{1.nonloc} and let
$\phi\in C_0^\infty(\R^d)$ with $\operatorname{supp}(\phi)\subset B_R$,
$\zeta\in C^0([0,T])$ be test functions, where $B_R$ is a ball around the origin
with radius $R>0$. Then the weak formulation of \eqref{1.nonloc} reads as
\begin{equation}\label{3.weak}
\begin{aligned}
  \int_0^T\langle\pa_t u_{\eta.i},\phi\rangle\zeta(t)\d t
	&= -\int_0^T\int_{\R^d} u_{\eta,i}\na U_i\cdot\na\phi\zeta(t)\d x\d t \\
  &\phantom{xx}{}-\int_0^T\int_{\R^d}\big(\sigma_i\na u_{\eta,i} 
	+ \na(u_{\eta,i}K_i(u_\eta))\big)
	\cdot\na\phi\zeta(t)\d x\d t,
\end{aligned}
\end{equation}
where $\langle\cdot,\cdot\rangle$ is the duality pairing between $H^{-1}(\R^d)$ and
$H^1(\R^d)$ and $K_i(u) =\sum_{j=1}^n f_\eta(B_{ij}^\eta*u_j)$. 
We want to perform the limit $\eta\to 0$.
By the uniform estimate of Theorem \ref{thm.nonloc}, there exists a subsequence,
which is not relabeled, such that $u_\eta\rightharpoonup u$ weakly in 
$L^2(0,T;H^{s+1}(\R^d))$ and weakly* in $L^\infty(0,T;H^s(\R^d))\subset 
L^\infty(0,T;L^\infty(\R^d))$ as $\eta\to 0$. Our aim is to prove that
$u$ is a weak solution to \eqref{1.skt}. 

It follows from the proof of Lemma 7 in \cite{CDJ19} that
$$
  B_{ij}^\eta * \na u_{\eta,j}\rightharpoonup a_{ij}\na u_j\quad
	\mbox{weakly in }L^2(0,T;L^2(\R^d)).
$$
We claim that $f_\eta(B_{ij}^\eta*u_{\eta,j})\to f(a_{ij}u_j)$ 
strongly in $L^2(0,T;L^2(B_R))$. 
First, we observe that $u\in L^\infty(0,T;L^\infty(\R^d))$.
The weak formulation \eqref{3.weak} gives
\begin{align*}
  \|\pa_t u_{\eta,i}\|_{L^2(0,T;H^{-1}(B_R))}
	&\le \|u_{\eta,i}\|_{L^2(0,T;L^2(\R^d))}\|\na U_i\|_{L^\infty(B_R)}
	+ \sigma_i\|\na u_{\eta,i}\|_{L^2(0,T;L^2(\R^d))} \\
	&\phantom{xx}{}+ \|\na u_{\eta,i}\|_{L^2(0,T;L^2(\R^d))}
	\|K_i(u_\eta)\|_{L^\infty(0,T;L^\infty(\R^d))} \\
  &\phantom{xx}{}+ \|u_{\eta,i}\|_{L^2(0,T;L^2(\R^d))}
	\|\na K_i(u_\eta)\|_{L^\infty(0,T;L^\infty(\R^d))}.
\end{align*}
Because of
\begin{align*}
  \|K_i(u_\eta)\|_{L^\infty(0,T;L^\infty(\R^d))}
	&\le \sum_{j=1}^n\|f_\eta(B_{ij}^\eta*u_{\eta,j})\|_{L^\infty(0,T;L^\infty(\R^d))}
	\le C\|f\|_{L^\infty(I)}, \\
	\|\na K_i(u_\eta)\|_{L^\infty(0,T;L^\infty(\R^d))}
	&\le \sum_{j=1}^n\|f'_\eta(B_{ij}^\eta*u_{\eta,j})\|_{L^\infty(0,T;L^\infty(\R^d))}
  \|B_{ij}^\eta*\na u_{\eta,j}\|_{L^\infty(0,T;L^\infty(\R^d))} \\
	&\le C\|f'\|_{L^\infty(I)}\|\na u_\eta\|_{L^\infty(0,T;L^\infty(\R^d))}
	\le C\|u_0\|_{H^s(\R^d)},
\end{align*}
we obtain a uniform bound for $\pa_t u_{\eta,i}$ in $L^2(0,T;H^{-1}(B_R))$
(the bound might depend on $R$). In particular, up to a subsequence, as $\eta\to 0$,
$$
  \pa_t u_{\eta,i}\rightharpoonup \pa_t u_i\quad\mbox{weakly in }L^2(0,T;H^{-1}(B_R)).
$$
Since $u_\eta$ is uniformly bounded in
$L^2(0,T;H^1(B_R))$, the Aubin--Lions lemma implies the existence of a subsequence
(not relabeled) such that
$$
  u_{\eta,i}\to u_i \quad\mbox{strongly in }L^2(0,T;L^2(B_R)).
$$
We use the Lipschitz continuity of $f=f_\eta$ on $I$  to infer that
\begin{align*}
  \|f_\eta(&B_{ij}^\eta * u_{\eta,j})-f(a_{ij}u_j)\|_{L^2(0,T;L^2(B_R))} \\
	&\le C\|B_{ij}^\eta*(u_{\eta,j}-u_j) 
	+ B_{ij}^\eta*u_j - a_{ij}u_j\|_{L^2(0,T;L^2(B_R))} \\
	&\le C\|B_{ij}^\eta\|_{L^1(\R^d)}\|u_{\eta,j}-u_j\|_{L^2(0,T;L^2(B_R))}
	+ \|B_{ij}^\eta*u_j-a_{ij}u_j\|_{L^2(0,T;L^2(B_R))} \to 0.
\end{align*}
This shows the claim. In a similar way, it follows from the Lipschitz continuity
of $f'_\eta$ that $f'_\eta(B_{ij}^\eta*u_{\eta,j})\to f'(a_{ij}u_j)$ 
strongly in $L^2(0,T;L^2(B_R))$.

The previous convergences allow us to perform the limit $\eta\to 0$ in \eqref{3.weak},
leading to
$$
  \int_0^T\langle\pa_t u_i,\phi\rangle\zeta(t)\d t
	= -\int_0^T\int_{\R^d}u_i\na U_i\cdot\na\phi\zeta(t)\d x\d t
	- \int_0^T\int_{\R^d}\na F_i(u)\cdot\na\phi\zeta(t)\d x\d t,
$$
where $F_i(u)=u_i(\sigma_i+\sum_{j=1}^n f(a_{ij}u_j))$. 
Moreover, $u_{i}(0)=u_{0,i}$ in $B_R$
for any $R>0$. Thus, $u$ is a weak solution to \eqref{1.skt}. Standard estimates
show that $u$ is the unique solution, again choosing $\eps>0$ sufficiently small.

{\em Step 2: Convergence rate.} We take the difference of \eqref{1.nonloc} and
\eqref{1.skt}, multiply the resulting equation by $u_{\eta,i}-u_i$, integrate
over $(0,\tau)\times\R^d$ for any $\tau>0$, and integrate by parts:
\begin{align}
  \frac12\int_{\R^d}&(u_{\eta,i}-u_i)(\tau)^2\d x
	+ \sigma_i\int_0^\tau\int_{\R^d}|\na(u_{\eta,i}-u_i)|^2\d x\d t
	= \frac12\int_0^\tau\int_{\R^d}\Delta U_i(u_{\eta,i}-u_i)^2\d x\d t \nonumber \\
	&{}- \int_0^\tau\int_{\R^d}\na\sum_{j=1}^n\big(u_{\eta,i}
	f_\eta(B_{ij}^\eta*u_{\eta,j})
	- u_if(a_{ij}u_j)\big)\cdot\na(u_{\eta,i}-u_i)\d x\d t. \label{3.aux0}
\end{align}
The first integral on the right-hand side is nonpositive since $\Delta U_i=-d$.
We split the second integral into three parts:
\begin{equation}\label{3.aux}
  -\int_0^\tau\int_{\R^d}\sum_{j=1}^n\na\big(u_{\eta,i}f_\eta(B_{ij}^\eta*u_{\eta,j})
	- u_if(a_{ij}u_j)\big)\cdot\na(u_{\eta,i}-u_i)\d x\d t
	= J_1 + J_2 + J_3,
\end{equation}
where
\begin{align*}
  J_1 &= -\int_0^\tau\int_{\R^d}\sum_{j=1}^n\na\big((u_{\eta,i}-u_i)
	f_\eta(B_{ij}^\eta*u_{\eta,j})\big)\cdot\na(u_{\eta,i}-u_i)\d x\d t, \\
	J_2 &= -\int_0^\tau\int_{\R^d}\sum_{j=1}^n\na\big(u_i\big(
	f_\eta(B_{ij}^\eta*u_{\eta,j})
	- f_\eta(a_{ij}u_{\eta,j})\big)\big)\cdot\na(u_{\eta,i}-u_i)\d x\d t, \\
  J_3 &= -\int_0^\tau\int_{\R^d}\sum_{j=1}^n\na\big(u_i\big(f_\eta(a_{ij}u_{\eta,j})
	- f(a_{ij}u_j)\big)\big)\cdot\na(u_{\eta,i}-u_i)\d x\d t.
\end{align*}

We start with the estimate of $J_1$. The families $(B_{ij}^\eta*u_{\eta,j})$ and
$(B_{ij}^\eta*\na u_{\eta,j})$ are bounded in $L^\infty(0,T;L^\infty(\R^d))$.
Using $\|f_\eta\|_{L^\infty(I)}=\|f\|_{L^\infty(I)}\le\eps$ and Young's inequality,
we have
\begin{align}
  J_1 &\le \|f_\eta(B_{ij}^\eta*u_{\eta,j})\|_{L^\infty(0,T;L^\infty(\R^d))}
	\int_0^\tau\|\na(u_{\eta,i}-u_i)\|_{L^2(\R^d)}^2\d t \nonumber \\
	&\phantom{xx}{}
	+ \int_0^\tau\|u_{\eta,i}-u_i\|_{L^2(\R^d)}\|f'_\eta(B_{ij}^\eta*u_{\eta,j})
	\|_{L^\infty(0,T;L^\infty(\R^d))} \nonumber \\
	&\phantom{xxxx}{}\times\|B_{ij}^\eta*\na u_{\eta,j}\|_{L^\infty(0,T;L^\infty(\R^d))} 
  \|\na(u_{\eta,i}-u_i)\|_{L^2(\R^d)}\d t \nonumber \\
	&\le \bigg(\frac{\sigma_i}{4}+\eps\bigg)\int_0^\tau
	\|\na(u_{\eta,i}-u_i)\|_{L^2(\R^d)}^2\d t
	+ C(\sigma_i)\int_0^\tau\|u_{\eta,i}-u_i\|_{L^2(\R^d)}^2\d t. \label{3.J1}
\end{align}

Next, we estimate $J_2=J_{21}+J_{22}$, where
\begin{align*}
  J_{21} &= -\int_0^\tau\int_{\R^d}\na u_i\sum_{j=1}^n
	\big(f_\eta(B_{ij}^\eta*u_{\eta,j})-f_\eta(a_{ij}u_{\eta,j})\big)
	\cdot\na(u_{\eta,i}-u_i)\d x\d t, \\
  J_{22} &= -\int_0^\tau\int_{\R^d}u_i\sum_{j=1}^n\big(f'_\eta(B_{ij}^\eta*u_{\eta,j})
	B_{ij}^\eta*\na u_{\eta,j}
	- f'_\eta(a_{ij}u_{\eta,j})a_{ij}\na u_{\eta,j}\big)\cdot\na(u_{\eta,i}-u_i)\d x\d t.
\end{align*}
It follows that
\begin{align*}
  J_{21} &\le \|\na u_i\|_{L^\infty(0,T;L^\infty(\R^d))}
	\sum_{j=1}^n\int_0^\tau\|f_\eta(B_{ij}^\eta*u_{\eta,j})-f_\eta(a_{ij}u_{\eta,j})
	\|_{L^2(\R^d)}
  \|\na(u_{\eta,i}-u_i)\|_{L^2(\R^d)}\d t \\
	&\le \frac{\sigma_i}{8}\int_0^\tau\|\na(u_{\eta,i}-u_i)\|_{L^2(\R^d)}^2\d t
	+ C\sum_{j=1}^n\int_0^\tau\|f_\eta(B_{ij}^\eta*u_{\eta,j})-f_\eta(a_{ij}u_{\eta,j})
	\|^2_{L^2(\R^d)}\d t.
\end{align*}
Since both $B_{ij}^\eta*u_{\eta,j}$ and $u_{\eta,j}$ are uniformly bounded in
$L^\infty(0,T;L^\infty(\R^d))$, we can choose $\eta>0$ sufficiently small
such that $f=f_\eta$ on $I$. On that interval, $f$ is
Lipschitz continuous uniformly in $\eta$. We use this information in 
$$
  \bigg|\int_{\R^d}\big(f_\eta(B_{ij}^\eta*u_{\eta,j})-f_\eta(a_{ij}u_{\eta,j})
	\big)g(x)\d x\bigg|
	\le C\int_{\R^d}\big|B_{ij}^\eta*u_{\eta,j}-a_{ij}u_{\eta,j}\big||g(x)|\d x,
$$
where $g\in L^2(\R^d)$. Recalling that $\operatorname{supp}(B_{ij}^\eta)\subset
B_\eta(0)$ and $a_{ij}=\int_{B_\eta}B_{ij}^\eta\d x$, we obtain
\begin{align*}
  \bigg|\int_{\R^d}&\big(f_\eta(B_{ij}^\eta*u_{\eta,j})-f_\eta(a_{ij}u_{\eta,j})
	\big)g(x)\d x\bigg| \\
	&\le C\int_{\R^d}\bigg|\int_{B_\eta}B_{ij}^\eta(y)\big(u_{\eta,j}(x-y)-u_{\eta,j}(x)
	\big)\d y\bigg||g(x)|\d x \\
	&\le C\int_{\R^d}\int_{B_\eta}|B_{ij}^\eta(y)|\bigg(
	\int_0^1|\na u_{\eta,j}(x-ry)|\eta\d r\bigg)\d y|g(x)|\d x \\
	&= C\eta\int_0^1\int_{B_\eta}|B_{ij}^\eta(y)|\bigg(\int_{\R^d}|\na u_{\eta,j}(x-ry)|
	|g(x)|\d x\bigg)\d y\d r \\
	&\le C\eta\int_0^1\int_{B_\eta}|B_{ij}^\eta(y)|
	\|\na u_{\eta,j}(\cdot-ry)\|_{L^2(\R^d)}\|g\|_{L^2(\R^d)}\d y\d r \\
	&\le C\eta\int_{B_\eta}|B_{ij}^\eta(y)|\d y\|\na u_{\eta,j}\|_{L^2(\R^d)}
	\|g\|_{L^2(\R^d)} \le C\eta\|g\|_{L^2(\R^d)}.
\end{align*}
By duality, we find that
$$
  J_{21}\le \frac{\sigma_i}{8}\int_0^\tau\|\na(u_{\eta,i}-u_i)\|_{L^2(\R^d)}^2\d t
	+ C\eta^2.
$$

The integral $J_{22}$ is split into $J_{22}=J_{221}+J_{222}$, where
\begin{align*}
  J_{221} &= -\int_0^\tau\int_{\R^d}u_i\sum_{j=1}^n f'_\eta(B_{ij}^\eta*u_{\eta,j})
	\big(B_{ij}^\eta*\na u_{\eta,j} - a_{ij}\na u_{\eta,j}\big)
	\cdot\na(u_{\eta,i}-u_i)\d x\d t, \\
  J_{222} &= -\int_0^\tau\int_{\R^d}u_i\sum_{j=1}^n\big(f'_\eta(B_{ij}^\eta*u_{\eta,j})
	- f'_\eta(a_{ij}u_{\eta,j})\big)a_{ij}\na u_{\eta,j}
	\cdot\na(u_{\eta,i}-u_i)\d x\d t.
\end{align*}
We infer from the uniform boundedness of $B_{ij}^\eta*u_{\eta,j}$ in 
$L^\infty(0,T;L^\infty(\R^d))$ and the fact that $f'_\eta=f'$ on $I$
for sufficiently small $\eta>0$ that
\begin{align*}
  J_{221} &\le \frac{\sigma_i}{16}\int_0^\tau\|\na(u_{\eta,i}-u_i)\|_{L^2(\R^d)}^2\d t
	+ C\red{\sum_{j=1}^n}\int_0^T\|B_{ij}^\eta*\na u_{\eta,j} - a_{ij}\na u_{\eta,j}
	\|_{L^2(\R^d)}^2\d t \\
	&\le \frac{\sigma_i}{16}\int_0^\tau\|\na(u_{\eta,i}-u_i)\|_{L^2(\R^d)}^2\d t
	+ C\eta^2\red{\sum_{j=1}^n}\int_0^\tau\|D^2 u_{\eta,j}\|_{L^2(\R^d)}^2\d t,
\end{align*}
where we estimated the difference $B_{ij}^\eta*\na u_{\eta,j} - a_{ij}\na u_{\eta,j}$
similarly as for $J_{21}$. Furthermore, the Lipschitz continuity of $f'_\eta=f'$ on $I$
leads to
\begin{align*}
  J_{222} &\le C\red{\sum_{j=1}^n}
	\int_0^\tau\|u_i\|_{L^\infty(\R^d)}\|B_{ij}^\eta*u_{\eta,j}
	- a_{ij}u_{\eta,j}\|_{L^2(\R^d)}\|\na u_{\eta,j}\|_{L^\infty(\R^d)}
  \|\na(u_{\eta,i}-u_i)\|_{L^2(\R^d)}\d t \\
	&\le \frac{\sigma_i}{16}\int_0^\tau\|\na(u_{\eta,i}-u_i)\|_{L^2(\R^d)}^2\d t
	+ C\eta^2\red{\sum_{j=1}^n}\int_0^\tau\|\na u_{\eta,j}\|_{L^2(\R^d)}^2\d t.
\end{align*}
Summarizing these estimates, we infer that
$$
  J_{22} \le \frac{\sigma_i}{8}\int_0^\tau\|\na(u_{\eta,i}-u_i)\|_{L^2(\R^d)}^2\d t
  + C\eta^2,
$$
and combining the estimate for $J_{21}$ and $J_{22}$,
\begin{equation}\label{3.J2}
  J_2 \le \frac{\sigma_i}{4}\int_0^\tau\|\na(u_{\eta,i}-u_i)\|_{L^2(\R^d)}^2\d t
  + C\eta^2.
\end{equation}

It remains to estimate $J_3=J_{31}+J_{32}$, where
\begin{align*}
  J_{31} &= -\int_0^\tau\int_{\R^d}\sum_{j=1}^n\big(f_\eta(a_{ij}u_{\eta,j})
	- f(a_{ij}u_j)\big)\na u_i\cdot\na(u_{\eta,i}-u_i)\d x\d t, \\
  J_{32} &= -\int_0^\tau\int_{\R^d}u_i\sum_{j=1}^n\big(f'_\eta(a_{ij}u_{\eta,j})
	a_{ij}\na u_{\eta,j} - f'(a_{ij}u_j)a_{ij}\na u_j\big)
	\cdot\na(u_{\eta,i}-u_i)\d x\d t.
\end{align*}
Similar arguments as above yield
\begin{align*}
  J_{31} &\le \frac{\sigma_i}{8}\int_0^\tau\|\na(u_{\eta,i}-u_i)\|_{L^2(\R^d)}^2\d t
	+ C\int_0^\tau\|\na u_i\|_{L^\infty(\R^d)}^2\|u_{\eta}-u\|_{L^2(\R^d)}^2\d t \\
	&\le \frac{\sigma_i}{8}\int_0^\tau\|\na(u_{\eta,i}-u_i)\|_{L^2(\R^d)}^2\d t
  + C\int_0^\tau\|u_{\eta}-u\|_{L^2(\R^d)}^2\d t.
\end{align*}
The second term $J_{32}$ is again split into two parts, $J_{32}=J_{321}+J_{322}$, where
\begin{align*}
  J_{321} &= -\int_0^\tau\int_{\R^d}u_i\sum_{j=1}^n \big(f'_\eta(a_{ij}u_{\eta,j})
	- f'_\eta(a_{ij}u_j)\big)a_{ij}\na u_{\eta,j}\cdot\na(u_{\eta,i}-u_i)\d x\d t, \\
  J_{322} &= -\int_0^\tau\int_{\R^d}u_i\sum_{j=1}^n a_{ij}\big(f'_\eta(a_{ij}u_j)
	\na u_{\eta,j} - f'(a_{ij}u_j)\na u_j\big)\cdot\na(u_{\eta,i}-u_i)\d x\d t.
\end{align*}
Using the Lipschitz continuity again, $f'_\eta=f'$ on $I$, and $|f'|\le\eps$, 
we deduce that
\begin{align*}
  J_{321} &\le C\|u_i\|_{L^\infty(0,T;L^\infty(\R^d)}\int_0^\tau\sum_{j=1}^n
	\|\na u_{\eta,j}\|_{L^\infty(\R^d)}
	\|u_{\eta,j}-u_j\|_{L^2(\R^d)}\|\na(u_{\eta,i}-u_i)\|_{L^2(\R^d)}\d t \\
  &\le \frac{\sigma_i}{8}\int_0^\tau\|\na(u_{\eta,i}-u_i)\|_{L^2(\R^d)}^2\d t
	+ C\int_0^\tau\|u_\eta-u\|_{L^2(\R^d)}^2, \\
	J_{322} &\le C\int_0^\tau\sum_{j=1}^n
	\|f'(a_{ij}u_j)\|_{L^\infty(\R^d)}\|\na(u_{\eta,j}-u_j)\|_{L^2(\R^d)}
	\|\na(u_{\eta,i}-u_i)\|_{L^2(\R^d)}\d t \\
	&\le C\eps\int_0^\tau\|\na(u_\eta-u)\|_{L^2(\R^d)}^2\d t.
\end{align*}
This shows that
$$
  J_{32} \le \bigg(\frac{\sigma_i}{8}+C\eps\bigg)
	\int_0^\tau\|\na(u_{\eta,i}-u_i)\|_{L^2(\R^d)}^2\d t
  + C\int_0^\tau\|u_\eta-u\|_{L^2(\R^d)}^2.
$$
Summarizing the estimate for $J_{31}$ and $J_{32}$, we arrive at
\begin{equation}\label{3.J3}
  J_3 \le \bigg(\frac{\sigma_i}{4}+C\eps\bigg)
	\int_0^\tau\|\na(u_{\eta,i}-u_i)\|_{L^2(\R^d)}^2\d t
  + C\int_0^\tau\|u_\eta-u\|_{L^2(\R^d)}^2\d t.
\end{equation}

Finally, putting together the estimates \eqref{3.J1}, \eqref{3.J2}, and \eqref{3.J3},
we infer from \eqref{3.aux} that
\begin{align*}
  \bigg|\int_0^\tau & \int_{\R^d}\sum_{j=1}^n\na\big(u_{\eta,i}
	f_\eta(B_{ij}^\eta*u_{\eta,j})
	- u_if(a_{ij}u_j)\big)\cdot\na(u_{\eta,i}-u_i)\d x\d t\bigg| \\
  &\le \bigg(\frac{3\sigma_i}{4} + C\eps\bigg)\int_0^\tau
	\|\na(u_{\eta,j}-u_j)\|_{L^2(\R^d)}^2\d t
	+ C\int_0^\tau\|u_{\eta}-u\|_{L^2(\R^d)}^2\d t + C\eta^2.
\end{align*}
This is the desired estimate for the last integral in \eqref{3.aux0}. We conclude 
for sufficiently small $\eps>0$ and after summation over $i=1,\ldots,n$ that
$$
  \|(u_\eta-u)(\tau)\|_{L^2(\R^d)}^2 + \sigma_{\rm min}C\int_0^\tau
	\|\na(u_\eta-u)\|_{L^2(\R^d)}^2\d t
	\le C\int_0^\tau\|u_{\eta}-u\|_{L^2(\R^d)}^2\d t + C\eta^2.
$$
The proof ends after applying Gronwall's inequality.


\section{Links between the SDEs and PDEs}\label{sec.link}

We show that the density function $\widehat{u}$ from Proposition \ref{prop.ex} 
coincides with the unique weak solution $u$ to \eqref{1.skt}.

\begin{theorem}
Let the assumptions of Theorem \ref{thm.loc} hold. Let $\widehat{X}_i$ for
$i=1,\ldots,n$ be the square-integrable process solving \eqref{1.ml} with 
density function $\widehat{u}_i$ and let $u_i$ be the unique weak solution to 
\eqref{1.skt}. Then $\widehat{u}=(\widehat u_1,\ldots,\widehat u_n)$ 
solves the linear equation
\begin{equation}\label{4.lin}
  \pa_t\widehat{u}_i = \diver(\widehat{u}_i\na U_i)
	+ \Delta\bigg(\sigma_i\widehat{u}_i + \widehat{u}_i\sum_{j=1}^n f(a_{ij}u_j)\bigg)
	\quad\mbox{in }\R^d,\ i=1,\ldots,n,
\end{equation}
in the weak integrable sense, i.e.
\begin{align*}
  \int_{\R^d}&\widehat{u}_i(t)\phi(t)\d x - \int_{\R^d}u_{0,i}\phi(0)\d x
	- \int_0^t\int_{\R^d}\widehat{u}_i\pa_t\phi\d x\d s \\
	&= -\int_0^t\int_{\R^d}\widehat{u}_i\na U_i\cdot\na\phi\d x\d t
	+ \int_0^t\int_{\R^d}\widehat{u}_i\bigg(\sigma_i+\sum_{j=1}^n f(a_{ij}u_j)\bigg)
	\Delta\phi\d x\d s
\end{align*}
for all $\phi\in C_0^\infty([0,\infty)\times\R^d)$ and $t>0$, 
where we assume that the initial datum $\widehat{u}_i(0)=u_{0,i}$ fulfils
\begin{equation}\label{4.condIni}
 \int_{\R^d}u_{0,i}(x)\d x = 1, \quad \int_{\R^d}u_{0,i}(x)|x|^2\d x < \infty.
\end{equation}
Additionally, $\widehat{u}=u$ in $(0,\infty)\times\R^d$, $u_i\ge 0$, and 
\eqref{4.condIni} is fulfilled \red{for $u_i$ instead of $u_{0,i}$} 
for almost all $t>0$ and all $i=1,\ldots,n$.
\end{theorem}

\begin{proof}
Since $\widehat{X}_{k,i}$
depends on $k$ only via the initial data $\xi_i^k$ with the same law $u_{0,i}$,
we can omit the index $k$. Let $\phi\in C_0^\infty([0,\infty)\times\R^d)$
and set $F_i(u)=\sigma_i+\sum_{j=1}^n f(a_{ij}u_j)$. By It\^o's lemma, we obtain
\begin{align}
  \phi(&t,\widehat{X}_i(t)) = \phi(0,\xi_i) + \int_0^t\pa_t\phi(s,\widehat{X}_i(s))\d s
	- \int_0^t\na U_i(s)\cdot\na\phi(s,\widehat{X}_i(s))\d s \nonumber \\
	&\phantom{xx}{}
	+ \int_0^t F_i\big(u(\widehat{X}_i(s))\big)\Delta\phi(s,\widehat{X}_i(s))\d s
	+ \int_0^t F_i\big(u(\widehat{X}_i(s))\big)^{1/2}\na\phi(s,X(s))\cdot\d W_i(s).
	\label{4.ito}
\end{align}
We claim that 
the density function $\widehat{u}_i:[0,\infty)\to \mathcal{P}_2(\R^d)$,
where $\mathcal{P}_2(\R^d)$ is the space of all density functions with
finite second moment, is continuous with respect to the 2-Wasserstein distance
$W_2$.
Indeed, since $\widehat{X}_i$ is square-integrable, we have $\widehat{u}_i(t)\in
\mathcal{P}_2(\R^d)$ for almost all $t>0$ and the limit $s\to t$ 
in the Wasserstein distance leads to
\begin{align*}
  W_2(\widehat{u}_i(t),\widehat{u}_i(s))
	&= \inf\big\{\big(\E(|Y_t-Y_s|^2)\big)^{1/2}:\
	\operatorname{Law}(Y_t)=\widehat{u}_i(t),\
	\operatorname{Law}(Y_s)=\widehat{u}_i(s)\big\} \\
	&\le \big(\E(|\widehat{X}_i(t)-\widehat{X}_i(s)|^2)\big)^{1/2}\to 0,
\end{align*}
using the facts that $\widehat{X}_i$ is continuous in time and has bounded second
moments. This shows the claim. We conclude that the point evaluation $\widehat{u}_i(t)$
is well defined.

The previous argumentation shows that we can apply the expectation to \eqref{4.ito}
to obtain
\begin{align*}
  \int_{\R^d}&\widehat{u}_i(t)\phi(t)\d x = \int_{\R^d}u_{0,i}\phi(0)\d x
	+ \int_0^t\int_{\R^d}\widehat{u}_i(s)\pa_t\phi(s)\d x\d s \\
	&{} -\int_0^t\int_{\R^d}\widehat{u}_i(s)\na U_i\cdot\na\phi(s)\d x\d s
	+ \int_0^t\int_{\R^d}\widehat{u}_i(s)F_i(u(s))\Delta \phi(s)\d x\d s.
\end{align*}
This is the very weak formulation of \eqref{4.lin}, showing the first part of
the theorem.

Next, we verify that the solution to \eqref{4.lin} is unique. More precisely,
we take $u_0=0$ and show that $\widehat{u}_i(t)=0$ for almost all $t>0$.
The statement is usually proved by a duality argument. However, the
coefficients of the dual problem associated to \eqref{4.lin} are not regular
enough such that we need to regularize it. As the proof is rather standard but
tedious, we only sketch the arguments. Let $\chi_k$ be a family of mollifiers
and consider the regularized dual backward problem on the ball $B_R$ around the origin
with radius $R>0$:
\begin{align*}
  & \pa_t w_{k,R} - \na U_i\cdot\na w_{k,R} + (\chi_k*F_i(u))\Delta w_{k,R} = 0
	\quad\mbox{in }B_R,\ 0<s<t, \\
  & w_{k,R}=0\quad\mbox{on }\pa B_R, \quad
	w_{k,R}(t)=g\in C_0^\infty(B_R)\quad\mbox{in } B_R.
\end{align*}
We extend the unique smooth solution $w_{k,R}$ to the whole space by 
setting $w_{k,R}=0$ on $\R^d\setminus B_R$. Since the extension may be not
smooth, we choose a cut-off function $\psi_R\in C^\infty(\R^d)$ and use
$w_{k,R}\psi_R$ as an admissible 
test function in the very weak formulation of \eqref{4.lin}.
Standard estimations give bounds for $w_{k,R}$ uniform in $k$ and $R$.
Then, passing to the limit $k\to\infty$, $R\to\infty$ in the weak formulation
shows that $\int_{\R^d}g(x)\widehat u_i(s,x)\d x=0$, and since $g$ was arbitrary,
we conclude that $\widehat{u}_i(s)=0$ for $0<s<t$.

The weak solution $u$ to \eqref{1.skt} is also a very weak solution to \eqref{4.lin}.
Therefore, by the previous uniqueness result, $\widehat{u}=u$. 
\end{proof}

Similar arguments lead to the following result
that relates the solutions $\overline{u}_\eta$ and $u_\eta$.

\begin{theorem}
Let the assumptions of Theorem \ref{thm.nonloc} hold and let $\eta>0$.
Let $\overline{X}^\eta_{k,i}$ for $i=1,\ldots,n$ and $k=1,\ldots,N$ be the 
square-integrable process solving \eqref{1.il} with density function
$\overline{u}_{\eta,i}$. Then $\overline{u}_{\eta}=(\overline{u}_{\eta,1},\ldots,
\overline{u}_{\eta,n})$ solves the linear problem
\begin{equation*}
  \pa_t\overline{u}_{\eta,i} = \diver(\overline{u}_{\eta,i}\na U_i)
	+ \Delta\bigg(\sigma_i\overline{u}_{\eta,i} 
	+ \overline{u}_{\eta,i}\sum_{j=1}^n f_\eta(B_{ij}^\eta * u_{\eta,j})\bigg)
	\quad\mbox{in }\R^d,\ i=1,\ldots,n,
\end{equation*}
with initial datum $\overline{u}_{\eta,i}(0)=u_{0,i}$, which fulfils 
\eqref{4.condIni}, where 
$u_{\eta,i}$ is the unique weak solution to \eqref{1.nonloc}.
Then $\overline{u}_\eta=u_\eta$ in $(0,\infty)\times\R^d$, 
$u_{\eta,i}\ge 0$, and
$$
  \int_{\R^d}u_{\eta,i}(x,t)\d x=1, \quad \int_{\R^d}u_{\eta,i}(x,t)|x|^2\d x<\infty
$$
for almost all $t>0$ and all $i=1,\ldots,n$.
\end{theorem}


\section{Proof of Theorem \ref{thm.conv}}\label{sec.conv}

The proof is split into two parts. We estimate first the square mean error
of the difference $X_{k,i}^{N,\eta}-\overline{X}_{k,i}^\eta$, where 
$\overline{X}_{k,i}^\eta$ is the solution to the intermediate system \eqref{1.il}.
\red{In fact, this error bound is a generalization of a result due to \cite{Szn91}.
Essential for this step are the facts that the Lipschitz constant of
$B_{ij}^\eta$ is of order $\eta^{-d-1}$, while the Lipschitz constant of $f_\eta$
is of order $\eta^{-\alpha}$. Second, we estimate the square mean error 
of the difference $\overline{X}_{k,i}^\eta-\widehat{X}_{k,i}$, based on 
an estimate of $f_\eta(B_{ij}^\eta*u_j)-f_\eta(a_{ij}u_j)$ in $L^2$, 
which is of the order of $\eta^{1-\alpha}$.}

\begin{lemma}\label{lem.conv1}
Let $X_{k,i}^{N,\eta}$ and $\overline{X}_{k,i}^\eta$ be the solutions to
\eqref{1.pl} and \eqref{1.ml}, respectively, in the sense of
Proposition \ref{prop.ex}. Under the assumptions of Theorem \ref{thm.conv}, 
there exists $\delta>0$, depending on $n$, $\sigma_{\rm min}$, and $T$, such that
if $\eta^{-2(d+1+\alpha)}\le \delta\log N$, where $\alpha\ge 0$ is fixed in 
Assumption {\rm (A4)}, we have
$$
  \sup_{k=1,\ldots,N}\E\bigg(\sum_{i=1}^n\sup_{0<s<T}
	\big|(X_{k,i}^{N,\eta}-\overline{X}_{k,i}^\eta)(s)\big|^2\bigg)
	\le C(T,n,\sigma_{\rm min})N^{-1+(T+1)C(n,\red{T},\sigma_{\rm min})\delta},
$$
where $C(T,n,\sigma_{\rm min})>0$ is a positive constant.
\end{lemma}

\begin{proof}
The process $D_{k,i}^{N,\eta}:=X_{k,i}^{N,\eta}-\overline{X}_{k,i}^\eta$ solves
\begin{equation}\label{4.D}
  D_{k,i}^{N,\eta}(s) = E_{1,i}(s) + E_{2,i}(s), \quad 0\le s\le T,
\end{equation}
where
\begin{align*}
  E_{1,i}(s) &= -\int_0^s\big(\na U_i(X_{k,i}^{N,\eta}(t))
	-\na U_i(\overline{X}_{k,i}^\eta(t))\big)\d t, \\
  E_{2,i}(s) &= \int_0^s(E_{21}(t)-E_{22}(t))\d W_i^k(t), \\
	E_{21}(t) &= \bigg(2\sigma_i + 2\sum_{j=1}^n f_\eta\bigg(\frac{1}{N}\!\!\!\!
	\sum_{\substack{\ell=1 \\ (\ell,j)\neq(k,i)}}^N\!\!\!\!
	B_{ij}^\eta(X_{k,i}^{N,\eta}(t)-X_{\ell,j}^{N,\eta}(t))\bigg)\bigg)^{1/2}, \\
	E_{22}(t) &= \bigg(2\sigma_i + 2\sum_{j=1}^n f_\eta\big(B_{ij}^\eta
	* u_{\eta,j}(t,\overline{X}_{k,i}^\eta(t))\big)\bigg)^{1/2}.
\end{align*}
We use the global Lipschitz continuity of $\na U_i$ and the Fubini theorem to
estimate the first term:
\begin{align*}
  \E\Big(\sup_{0<s<T} |E_{1,i}(s)|^2\Big)
	&\le C\red{T}
	\E\int_0^T\big|(X_{k,i}^{N,\eta}-\overline{X}_{k,i}^\eta)(s)\big|^2\d s \\
	&\le C\red{T}
	\int_0^T\E\Big(\sup_{0<s<t}|(X_{k,i}^{N,\eta}-\overline{X}_{k,i}^\eta)(s)|^2\Big)\d t.
\end{align*}
Summing over $i=1,\ldots,n$ and taking the supremum over $k=1,\ldots,N$ leads to
\begin{equation}\label{4.K1}
  \sup_{k=1,\ldots,N}\E\bigg(\sum_{i=1}^n\sup_{0<s<T} |E_{1,i}(s)|^2\bigg)
	\le C\red{T}\int_0^T\sup_{k=1,\ldots,N}\E\Big(\sup_{0<s<t}
	|(X_{k,i}^{N,\eta}-\overline{X}_{k,i}^\eta)(s)|^2\Big)\d t.
\end{equation}

Next, we apply the Burkholder--Davis--Gundy inequality 
\cite[Theorem 3.28]{KaSh91} to the second term $E_{2,i}$
and use the Lipschitz continuity of $x\mapsto(2\sigma_i+x)^{1/2}$ for $x\ge 0$:
\begin{align}
  \E\Big(&\sup_{0<s<T}|E_{2,i}(s)|^2\Big) \le C\E\int_0^T(E_{21}(t)-E_{22}(t))^2\d t 
	\nonumber \\
  &\le C\E\int_0^T\bigg[\sum_{j=1}^n f_\eta\bigg(\frac{1}{N}\!\!\!\!
	\sum_{\substack{\ell=1 \\ (\ell,j)\neq(k,i)}}^N\!\!\!\!
	B_{ij}^\eta(X_{k,i}^{N,\eta}(t)-X_{\ell,j}^{N,\eta}(t))\bigg) \nonumber \\
	&\phantom{xxxxxx}{}
	- \sum_{j=1}^n f_\eta\big(B_{ij}^\eta*u_{\eta,j}(t,\overline{X}_{k,i}^\eta(t))
	\big)\bigg]^2\d t \nonumber \\
	&= C\E\int_0^T\bigg[\sum_{j=1}^n(L^1_j(t)+L^2_j(t)+L^3_j(t))\bigg]^2\d t 
	\nonumber \\
	&\le C(n)\E\int_0^T\sum_{j=1}^n\big(L^1_j(t)^2 + L^2_j(t)^2 + L^3_j(t)^2\big)\d t,
	\label{4.aux2}
\end{align}
where 
\begin{align*}
  L^1_j(t) &= f_\eta\bigg(\frac{1}{N}\!\!\!\!
	\sum_{\substack{\ell=1 \\ (\ell,j)\neq(k,i)}}^N\!\!\!\! B_{ij}^\eta
	(X_{k,i}^{N,\eta}(t)-X_{\ell,j}^{N,\eta}(t))\bigg)
	- f_\eta\bigg(\frac{1}{N}\!\!\!\!
	\sum_{\substack{\ell=1 \\ (\ell,j)\neq(k,i)}}^N\!\!\!\!B_{ij}^\eta
  (\overline{X}_{k,i}^{\eta}(t)-X_{\ell,j}^{N,\eta}(t))\bigg), \\
	L^2_j(t) &= f_\eta\bigg(\frac{1}{N}\!\!\!\!
	\sum_{\substack{\ell=1 \\ (\ell,j)\neq(k,i)}}^N\!\!\!\! B_{ij}^\eta
	(\overline{X}_{k,i}^{\eta}(t)-X_{\ell,j}^{N,\eta}(t))\bigg)
	- f_\eta\bigg(\frac{1}{N}\!\!\!\!
	\sum_{\substack{\ell=1 \\ (\ell,j)\neq(k,i)}}^N\!\!\!\! B_{ij}^\eta
  (\overline{X}_{k,i}^{\eta}(t)-\overline{X}_{\ell,j}^{\eta}(t))\bigg), \\
	L^3_j(t) &= f_\eta\bigg(\frac{1}{N}\!\!\!\!
	\sum_{\substack{\ell=1 \\ (\ell,j)\neq(k,i)}}^N\!\!\!\! B_{ij}^\eta
	(\overline{X}_{k,i}^{\eta}(t)-\overline{X}_{\ell,j}^{\eta}(t))\bigg)
	- f_\eta\big(B_{ij}^\eta* u_{\eta,j}(t,\overline{X}_{k,i}^\eta(t))\big). \\
\end{align*}

We estimate these three terms separately. By construction, the Lipschitz 
constant of $f_\eta$ can be estimated by $L_f\le\eta^{-\alpha}$.
Moreover, the Lipschitz constant of $B_{ij}^\eta(x)=\eta^{-d}B_{ij}(|x|/\eta)$
is computed by $L_B=\max_{i,j=1,\ldots,n}\|\na B_{ij}^\eta\|_{L^\infty(\R^d)}
\le C\eta^{-d-1}$. This shows that
\begin{align*}
  |L^1_j(t)| &\le L_f\bigg|\frac{1}{N}\!\!\!\!
	\sum_{\substack{\ell=1 \\ (\ell,j)\neq(k,i)}}^N\!\!\!\!\big(B_{ij}^\eta
	(X_{k,i}^{N,\eta}(t)-X_{\ell,j}^{N,\eta}(t)) - B_{ij}^\eta
	(\overline{X}_{k,i}^\eta(t)-X_{\ell,j}^{N,\eta}(t))\big)\bigg| \\
  &\le L_f L_B\big|X_{k,i}^{N,\eta}(t)-\overline{X}_{k,i}^\eta(t)\big|
	\le \red{C}\eta^{-d-1-\alpha}\big|X_{k,i}^{N,\eta}(t)-\overline{X}_{k,i}^\eta(t)\big|.
\end{align*}
Therefore, by Fubini's theorem,
\begin{align}
  \E\int_0^T\sum_{j=1}^n |L^1_j(t)|^2\d t
	&\le C(n)\eta^{-2(d+1 +\alpha)}\E\int_0^T\big|X_{k,i}^{N,\eta}(t)
	-\overline{X}_{k,i}^\eta(t)\big|^2\d t \nonumber \\
	&\le C(n)\eta^{-2(d+1 +\alpha)}\int_0^T\sup_{k=1,\ldots,N}\E\Big(
	\sup_{0<s<t}\big|X_{k,i}^{N,\eta}(t)-\overline{X}_{k,i}^\eta(t)\big|^2\Big)\d t.
	\label{4.EL1}
\end{align}
We can estimate the second term $L^2_j(t)$ in a similar way, leading to
\begin{equation}\label{4.EL2}
  \E\int_0^T\sum_{j=1}^n L^2_j(t)^2\d t
	\le C(n)\eta^{-2(d+1 +\alpha)}\int_0^T\sup_{\ell=1,\ldots,N}
	\E\bigg(\sup_{0<s<t}\sum_{j=1}^n
	\big|X_{\ell,j}^{N,\eta}(t)-\overline{X}_{\ell,j}^\eta(t)\big|^2\bigg)\d t.
\end{equation}

The third term $L^3_j(t)$ has to be treated in a different way. First, we use the 
Lipschitz continuity of $f_\eta$ to find that 
$$
  L^3_j(t) \le \frac{C(n)}{N\eta^{\alpha}}\bigg|
	\sum_{\ell=1}^N
	\big(B_{ij}^\eta(\overline{X}_{k,i}^\eta - \overline{X}_{\ell,j}^\eta)
	- B_{ij}^\eta * u_{\eta,j}(\overline{X}_{k,i}^\eta)\big) 
	- \frac{1}{\eta^d}B_{ii}(0)\bigg|.
$$
This implies that
\begin{align}
  \E&\int_0^T \sum_{j=1}^n L^3_j(t)^2\d t 
	\le \frac{C(n,T)}{N^2\eta^{2(d+\alpha)}} \nonumber \\
  &{}+ \frac{C(n)}{N^2\eta^{2\alpha}}\sum_{j=1}^n\int_0^T\E\bigg(
	\!\sum_{\ell=1}^N
	\Big(B_{ij}^\eta\big(\overline{X}_{k,i}^\eta(t) - \overline{X}_{\ell,j}^\eta(t)\big)
	- B_{ij}^\eta * u_{\eta,j}(\overline{X}_{k,i}^\eta)\Big)\bigg)^2 \d t. \label{4.L3}
\end{align}

It remains to estimate the expectation. To this end, we introduce
$$
  D_{(k,i),(\ell,j)}(t) := B_{ij}^\eta(\overline{X}_{k,i}^\eta(t) 
	- \overline{X}_{\ell,j}^\eta(t)) 
	- B_{ij}^\eta * u_{\eta,j}(t,\overline{X}_{k,i}^\eta(t)), \quad
	(\ell,j)\neq(k,i).
$$
The processes $\overline{X}_{k,i}^\eta$ and $\overline{X}_{\ell,j}^\eta$
are independent, since for $i=j$, we are considering $N$ independent copies of
the same process and for $i\neq j$, the equation fulfilled by $\overline{X}_{k,i}^\eta$
does not depend on the process $\overline{X}_{\ell,j}^\eta$.
If $(k,i)\neq (\ell,j)$, $(k,i)\neq(m,j)$, and $\ell\neq m$, the processes 
$D_{(k,i),(\ell,j)}(t)$ and $D_{(k,i),(m,j)}(t)$ are orthogonal, since 
\begin{align*}
  \E\big(D_{(k,i),(\ell,j)}(t)D_{(k,i),(m,j)}(t)\big)
	&= \int_{\R^d}\bigg(\int_{\R^d}\int_{\R^d}B_{ij}^\eta(x-y)B_{ij}^\eta(x-z)
	u_{\eta,j}(t,y)u_{\eta,j}(t,z)\d y\d z \\
	&\phantom{xx}{}
	- 2\int_{\R^d}B_{ij}^\eta(x-y)u_{\eta,j}(t,y)(B_{ij}^\eta*u_{\eta,j})(t,y)\d y \\
	&\phantom{xx}{}+ (B_{ij}^\eta*u_{\eta,j})(t,x)(B_{ij}^\eta*u_{\eta,j})(t,x)
	\bigg)u_{\eta,i}(t,x)\d x = 0.
\end{align*}
Together with $\E(D_{(k,i),(\ell,j)})=0$, this shows that the processes 
$D_{(k,i),(\ell,j)}$ are uncorrelated.

However, if $(k,i)\neq (\ell,j)$, $(k,i)\neq(m,j)$, and $\ell=m$, 
the expectation does not vanish:
\begin{align*}
  \E\big(D_{(k,i),(\ell,j)}(t)^2\big) &= \int_{\R^d}\bigg((B^\eta_{ij}*u_{\eta,j})(t,x)
	(B_{ij}^\eta*u_{\eta,j})(t,x) + \int_{\R^d}\Big(B_{ij}^\eta(x-y)^2u_{\eta,j}(t,y) \\
	&\phantom{xx}{}
	- 2B_{ij}^\eta(x-y)u_{\eta,j}(t,y)(B_{ij}^\eta*u_{\eta,j}(t,x)\Big)\d y\bigg)
	u_{\eta,i}(t,x)\d x \\
  &= \int_{\R^d}\big(((B_{ij}^\eta)^2*u_{\eta,j})(t,x)
	- (B_{ij}^\eta*u_{\eta,j})(t,x)^2\big)
	u_{\eta,i}(t,x)\d x.
\end{align*}
This expression is independent of the particle index $k$ and $\ell$, it depends only
on the species numbers $i$ and $j$. The case $(k,i)=(\ell,j)$ can be treated in 
a similar way with the difference that, since $ D_{(k,i),(k,i)}(t) 
= \eta^{-d}B_{ii}(0) - B_{ii}^\eta\ast u_{i,\eta}(\overline{X}_{k,i}^\eta(t))$,
we obtain for $\mathbb{E}(D_{(k,i),(k,i)}(t)D_{(k,i),(m,j)}(t))$ an additional 
term of order $\eta^{-2d}$.
Hence, we infer from \eqref{4.L3} and the previous
computation that
\begin{align}
  \E&\int_0^T\sum_{j=1}^n L^3_j(t)^2\d t -\frac{C(n,T)}{N^2\eta^{2(d+\alpha)}}
	= \frac{C(n)}{N^2\eta^{2\alpha}}\sum_{j=1}^n\sum_{\ell=1}^N
	\int_0^T\E\big(D_{(k,i),(\ell,j)}(t)^2\big)\d t \nonumber \\
	&\le \frac{C(n)}{N\eta^{2\alpha}}\|u_{\eta,i}\|_{L^\infty(0,T;L^\infty(\R^d))} 
	\nonumber \\
	&\phantom{xx}{}\times
	\sum_{j=1}^n\int_0^T\bigg(\|(B_{ij}^\eta)^2*u_{\eta,j}\|_{L^1(\R^d)} 
	+ \|B_{ij}^\eta*u_{\eta,j}\|_{L^2(\R^d)}^2\bigg(1+\frac{1}{\eta^{2d}}\bigg)
	\bigg)\d t \nonumber \\
  &\le \frac{C(n)}{N\eta^{2\alpha}}\sum_{j=1}^n\int_0^T
	\bigg(\|B_{ij}^\eta\|_{L^2(\R^d)}^2
	\|u_{\eta,j}\|_{L^\infty(\R^d)} + \|B_{ij}^\eta\|_{L^1(\R^d)}^2
	\|u_{\eta,j}\|_{L^2(\R^d)}^2\bigg(1+\frac{1}{\eta^{2d}}\bigg)\bigg)\d t \nonumber \\
	&\le \frac{C(T,n)}{N\eta^{2(d + \alpha)}}, \label{4.EL3}
\end{align}
recalling that $\|B_{ij}^\eta\|_{L^2(\R^d)}\le C\eta^{-d/2}$ and
$\|B_{ij}^\eta\|_{L^1(\R^d)}=A_{ij}\le A$ and choosing $\eta<1$. 

Inserting estimates \eqref{4.EL1}, \eqref{4.EL2}, and \eqref{4.EL3} for
$L^m_j(t)$ ($m=1,2,3$) into \eqref{4.aux2}, we conclude that
\begin{align*}
  \sup_{k=1,\ldots,N}&\E\bigg(\sum_{i=1}^n\sup_{0<s<T}|E_{2,i}(s)|^2\bigg)
	\le \frac{C(T,n)}{N\eta^{2(d+\alpha)}} \\
	&\phantom{}{}+ C(n,\sigma_{\rm min})\eta^{-2(d+1+\alpha)}
	\int_0^T\sup_{k=1,\ldots,N}\E\Big(\sup_{0<s<t}
	\big|X_{k,i}^{N,\eta}(t)-\overline{X}_{k,i}^\eta(t)\big|^2\Big)\d t.
\end{align*}
We infer from \eqref{4.D}, estimate \eqref{4.K1}, and the previous estimate for
$E_{2,i}$ that
\begin{align*}
  S(T) &:=
	\sup_{k=1,\ldots,N}\E\bigg(\sum_{i=1}^n\sup_{0<s<t}|D_{k,i}^{N,\eta}(s)|^2\bigg) \\
	&\le \frac{C(T,n)}{N\eta^{2(d+\alpha)}}
	+ C(n,\sigma_{\rm min})(\eta^{-2(d+1+\alpha)}+\red{T})\int_0^T S(t)\d t.
\end{align*}
Note that the function $S$ is continuous because of the continuity of the paths
of $X_{k,i}^{N,\eta}$ and $\overline{X}_{k,i}^\eta$. Therefore, by Gronwall's
inequality, we have
$$
  S(T) \le \frac{C(T,n)}{N\eta^{2(d+\alpha)}}
	\exp\big(C(n,\red{T},\sigma_{\rm min})\eta^{-2(d+1+\alpha)}T\big).
$$
We choose $\delta>0$ such that $C(n,\red{T},\sigma_{\rm min})T\delta < 1$ and
$\eta>0$ such that $\eta^{-2(d+1+\alpha)}\le\delta\log N$. Then
$$
  S(T) \le \frac{1}{N}C(T,n)\exp\big(C(n,\red{T},\sigma_{\rm min})T
	\delta\log N\big)
	= C(T,n)N^{-1+C(n,\red{T},\sigma_{\rm min})T\delta}.
$$
This finishes the proof.
\end{proof}

Next, we prove an error estimate for the difference $\overline{X}_{k,i}^\eta
-\widehat{X}_{k,i}$.

\begin{lemma}\label{lem.conv2}
Let $\overline{X}_{k,i}^\eta$ and $\widehat{X}_{k,i}$ be the solutions to 
\eqref{1.il} and \eqref{1.ml} in the sense of
Proposition \ref{prop.ex}. Under the assumptions of Theorem \ref{thm.conv}, 
it holds for small $\eta>0$ that
$$
  \sup_{k=1,\ldots,N}\E\bigg(\sum_{i=1}^n\sup_{0<s<T}
	\big|(\overline{X}_{k,i}^{\eta}-\widehat{X}_{k,i})(s)\big|^2\bigg) 
	\le C(T,\sigma_{\rm min})\eta^{2(1-\alpha)}.
$$
\end{lemma}

\begin{proof}
Since we are considering $N$ independent copies, we can omit the particle index $k$.
Set $D_i^\eta(s):=\overline{X}_{k,i}^{\eta}(s)-\widehat{X}_{k,i}(s)$. Then,
similarly as in the proof of Lemma \ref{lem.conv1},
$D_i^\eta(s)=D_1(s)+D_2(s)$, where
\begin{align*}
  D_1(s) &= -\int_0^s\big(\na U_i(\overline{X}_i^\eta(t))-\na U_i(\widehat{X}_i(t))
	\big)\d t, \\
  D_2(s) &= \int_0^s\bigg[\bigg(2\sigma_i + 2\sum_{j=1}^n f_\eta\big(B_{ij}^\eta
	* u_{\eta,j}(\overline{X}_i^\eta)\big)\bigg)^{1/2} \\
	&\phantom{xx}{}
	- \bigg(2\sigma_i + 2\sum_{j=1}^n f\big(a_{ij}u_j(\widehat{X}_i)\big)\bigg)^{1/2}
	\bigg]\d W_i(t).
\end{align*}
We infer from the Lipschitz continuity of $\na U_i$ and Fubini's theorem that
\begin{align}
  \E\Big(\sup_{0<s<T}|D_1(s)|^2\Big)
	\le C\red{T}\E\bigg(\int_0^T\big|\overline{X}_{i}^{\eta}(s)-\widehat{X}_{i}(s)\big|^2
	\d s\bigg) 
	\le C\red{T}\int_0^T\E\Big(\sup_{0<s<t}|D_i^\eta(s)|^2\Big)\d t. \label{4.D1}
\end{align}
Similarly as in the proof of Lemma \ref{lem.conv1}, we use for $D_2$ the
Burkholder--Davis--Gundy inequality and the Lipschitz continuity of
$x\mapsto(2\sigma_i+x)^{1/2}$ on $[0,\infty)$ to obtain
\begin{align}
  \E\Big(\sup_{0<s<T}|D_2(s)|^2\Big)
	&\le C\E\int_0^T\bigg(\sum_{j=1}^n\big(f(a_{ij}u_j(\widehat{X}_i))
	- f_\eta(B_{ij}^\eta*u_{\eta,j}(\overline{X}_i^\eta))
	\big)\bigg)^2\d t \nonumber \\
	&\le C(n)(D_{21}+D_{22}+D_{23}+D_{24}), \label{4.D2}
\end{align}
where
\begin{align*}
  D_{21} &= \sum_{j=1}^n\E\int_0^T\big(f(a_{ij}u_j(\widehat{X}_i)) 
	- f_\eta(a_{ij}u_j(\widehat{X}_i))\big)^2\d t, \\
	D_{22} &= \sum_{j=1}^n\E\int_0^T\big(f_\eta(a_{ij}u_j(\widehat{X}_i)) 
	- f_\eta(B_{ij}^\eta*u_j(\widehat{X}_i))\big)^2\d t, \\
	D_{23} &= \sum_{j=1}^n\E\int_0^T\big(f_\eta(B_{ij}^\eta*u_j(\widehat{X}_i)) 
	- f_\eta(B_{ij}^\eta*u_{j}(\overline{X}_i^\eta))\big)^2\d t, \\
	D_{24} &= \sum_{j=1}^n\E\int_0^T\big(f_\eta(B_{ij}^\eta*u_{j}(\overline{X}_i^\eta)) 
	- f_\eta(B_{ij}^\eta*u_{\eta,j}(\overline{X}_i^\eta))\big)^2\d t.
\end{align*}

The first expression $D_{21}$ vanishes if $\eta>0$ is sufficiently small,
since then $f=f_\eta$ on the range of $a_{ij}u_j(\widehat{X}_i)$.
Using 
$$
  \|a_{ij}u_j-B_{ij}^\eta*u_j\|_{L^2(0,T;L^2(\R^d))}
  \le C\eta\|\na u_j\|_{L^2(0,T;L^2(\R^d))} \le C\eta,
$$
which was shown in the proof of Theorem \ref{thm.loc}, and the Lipschitz continuity 
of $f_\eta$ with Lipschitz constant less or equal $\eta^{-\alpha}$, 
we find that
\begin{align*}
  D_{22} &= \sum_{j=1}^n\int_0^T\int_{\R^d}\big(f_\eta(a_{ij}u_j)
	- f_\eta(B_{ij}^\eta*u_j)\big)^2 u_i\d x\d t \\
	&\le \eta^{-2\alpha}\sum_{j=1}^n\|u_i\|_{L^\infty(0,T;L^\infty(\R^d))}
	\|a_{ij}u_j-B_{ij}^\eta*u_j\|_{L^2(0,T;L^2(\R^d))}^2
  \le C(n)\eta^{2(1-\alpha)}.
\end{align*}
Thanks to the uniform boundedness of the family $B_{ij}^\eta*u_j$, we can choose 
$\eta>0$ sufficiently small, say $\eta\le\eta^*$ for some $\eta^*>0$, such that 
$f(B_{ij}^\eta*u_j)=f_\eta(B_{ij}^\eta*u_j)$ for $0<\eta\le\eta^*$.
Then, using Young's convolution inequality and the uniform estimate
$\|\na u_j\|_{L^\infty(0,T;L^\infty(\R^d))}\le C\|u_0\|_{H^s(\R^d)}$ from
Theorem \ref{thm.loc},
the third term $D_{23}$ is estimated as
\begin{align*}
  D_{23} &\le C(\eta^*)\sum_{j=1}^n
	\|\na(B_{ij}^\eta*u_j)\|_{L^\infty(0,T;L^\infty(\R^d)}
	\int_0^T\E\big(|\widehat{X}_i(t)-\overline{X}_i^\eta(t)|^2\big)\d t \\
  &\le C\sum_{j=1}^n\|\na u_j\|_{L^\infty(0,T;L^\infty(\R^d)}
	\int_0^T\E\big(|\widehat{X}_i(t)-\overline{X}_i^\eta(t)|^2\big)\d t \\
	&\le C\int_0^T\E\Big(\sup_{0<s<t}|D_i^\eta(s)|^2\Big)\d t.
\end{align*}
Finally, it follows from the error estimate for
$u-u_{\eta}$ from Theorem \ref{thm.loc} that
\begin{align*}
  D_{24} &\le C\sum_{j=1}^n\int_0^T\int_{\R^d}|B_{ij}^\eta*u_j
	- B_{ij}^\eta*u_{\eta,j}|^2 u_{\eta,i}\d x\d t \\
	&\le C\sum_{j=1}^n\|u_{\eta,i}\|_{L^\infty(0,T;L^\infty(\R^d))}
  \int_0^T\|B_{ij}^\eta\|_{L^1(\R^d)}^2\|u_j-u_{\eta,j}\|_{L^2(\R^d)}^2\d t \\
	&\le C(T)\eta^{2}.
\end{align*}

Inserting the estimates for $D_{21},\ldots,D_{24}$ into \eqref{4.D2}, we conclude that
$$
  \E\Big(\sup_{0<s<T}|D_2(s)|^2\Big)
	\le C(T,n)\eta^{2(1-\alpha)} 
	+ C(\red{T})\int_0^T\E\Big(\sup_{0<s<t}|D_i^\eta(s)|^2\Big)\d t.
$$
Together with estimate \eqref{4.D1} for $D_1(s)$ and recalling that
$D_i^\eta=D_1+D_2$, we arrive at
$$
  \E\Big(\sup_{0<s<T}|D_i^\eta(s)|^2\Big)
	\le C(T,n)\eta^{2(1-\alpha)} 
	+ C(\red{T})\int_0^T\E\Big(\sup_{0<s<t}|D_i^\eta(s)|^2\Big)\d t.
$$
The proof is finished after applying Gronwall's inequality and summing
over $i=1,\ldots,n$.
\end{proof}

Theorem \ref{thm.conv} now follows from Lemmas \ref{lem.conv1} and
\ref{lem.conv2} and the triangle inequality:
  \begin{align*}
  \sup_{k=1,\ldots,N}\E&\bigg(\sum_{i=1}^n\sup_{0<s<t}
	\big|X_{\eta,i}^{k,N}(s) - \widehat X_i^k(s)\big|^2\bigg) \\
	&\le \red{2}\sup_{k=1,\ldots,N}\E\bigg(\sum_{i=1}^n\sup_{0<s<t}
	\big|X_{\eta,i}^{k,N}(s) - \overline{X}_{\eta,i}^k(s)\big|^2\bigg) \\
	&\phantom{xx}{}+ \red{2}\sup_{k=1,\ldots,N}\E\bigg(\sum_{i=1}^n\sup_{0<s<t}
	\big|\overline{X}_{\eta,i}^k(s) - \widehat X_i^k(s)\big|^2\bigg) \\
  &\le C_1N^{-1+C_2\delta} + C_3\eta^{2(1-\alpha)}.
\end{align*}
The condition $\log N\ge \delta^{-1}\eta^{-2(d+1+\alpha)}$ is equivalent to
$N^{-1+C_2\delta}\le\exp((-\delta^{-1}+C_2)\eta^{-2(d+1+\alpha)})$.
We choose $\delta>0$ such that $-\delta^{-1}+C_2<0$ and observe that
exponential decay is always faster than algebraic decay to conclude that
$\exp((-\delta^{-1}+C_2)\eta^{-2(d+1+\alpha)})\le \eta^{2(1-\alpha)}$. This yields
$$
  \sup_{k=1,\ldots,N}\E\bigg(\sum_{i=1}^n\sup_{0<s<t}
	\big|X_{\eta,i}^{k,N}(s) - \widehat X_i^k(s)\big|^2\bigg)
	\le C_4\eta^{2(1-\alpha)},
$$
finishing the proof.


\section{Numerical tests}\label{sec.num}

In this section, we perform some numerical simulations of the particle system
\eqref{1.pl} in one space dimension, without environmental potential, 
and with linear function $f(x)=x$. 
We are interested in the numerical comparison of the solutions to the
particle systems \eqref{1.cdj2} and \eqref{1.pl} \red{in terms of the
segregation behavior}.
We explore the ability of both systems to model the segregation of the species.
Numerical tests for the associated cross-diffusion systems 
\eqref{1.skt} and \eqref{1.cdj} are work in progress.

We discretize the particle systems \eqref{1.cdj2} and \eqref{1.pl} by the
Euler--Maruyama scheme. Let $M\in\N$ and introduce the time steps 
$0<t_1<\cdots<t_M=T$ with $\triangle t_m=t_{m+1}-t_m$. 
We approximate $X_{k,i}^{N,\eta}(t_m)$ by $x_{m}^{k,i}$
and $Y_{k,i}^{N,\eta}(t_m)$ by $y_{m}^{k,i}$, defined by, respectively,
\begin{align*}
  x_{m+1}^{k,i} &= x_m^{k,i} + \bigg(2\sigma_i + \frac{2}{N}\sum_{j=1}^n
  \sum_{\ell=1}^N B_{ij}^\eta(x_m^{k,i}-x_m^{\ell,j})\bigg)^{1/2}\sqrt{\triangle t_m}
	w_m, \\
  y_{m+1}^{k,i} &= y_m^{k,i} - \sum_{j=1}^n\frac{1}{N}\sum_{\ell=1}^N 
	\na B_{ij}^\eta(y_m^{k,i}-y_\ell^{m,j})\triangle t_m + \sqrt{2\sigma_i\triangle t_m}
	z_m,
\end{align*}
with initial conditions $x_0^{i,k} = \xi_i^k$ and $y_0^{i,k}=\xi_i^k$, where
$\xi_i^k$ are iid random variables and $w_m$ and $z_m$ are normally distributed.
It is well known that the solutions to the Euler--Maruyama scheme converge to
the associated stochastic processes in the strong sense; see, e.g., 
\cite[Theorem 9.6.2]{KlPl92}. 

The numerical scheme is implemented in MATLAB using the parallel computing toolbox
to accelerate the simulations. 
The interaction potential is given by
$B(x)=\exp(-1/(1-x^2))$ for $|x|\le 1$ and $B(x)=0$ else. Then
$B_{ij}^\eta(x)=\eta^{-1}B(x/\eta)$. The numerical parameters are
$\triangle t=1/100$, $\eta=2$, $N=5000$ particles, $n_{\rm sim}=500$ simulations.

\subsection{Two species: nonsymmetric case} 

We consider a nonsymmetric diffusion matrix with $a_{11}=0$, $a_{12}=355$, 
$a_{21}=25$, $a_{22}=0$, and $\sigma_1=1$, $\sigma_2=2$. 
The initial data are Gaussian distributions with 
mean $-1$ (for species $i=1$) and $1$ (for species $i=2$)
and variance 2. Figure \ref{fig.nsymm}
shows the approximate densities of both species (histogram)
for systems \eqref{1.pl} and \eqref{1.cdj2} at time $t=2$. 
We observe a segregation of the densities in both models. In the population 
system \eqref{1.pl}, species 1 develops two clusters because
of the very different ``population pressure'' parameters $a_{12}=355$ and
$a_{21}=25$, while species 2 develops only one cluster around $x=0$; see
Figure \ref{fig.nsymm} left.
The segregation effect is stronger in the particle system \eqref{1.cdj2} in the
sense that both species avoid each other as far as possible; see
Figure \ref{fig.nsymm} right. This is not surprising since the diffusion of
system \eqref{1.pl} is generally larger than that one of system \eqref{1.cdj2}.
The numerical results confirm the segregation property defined in \cite{BGHP85}.
Indeed, this work considers the cross-diffusion system \eqref{1.cdj2} with
$\sigma_1=\sigma_2=0$ and $a_{11}=a_{12}=a_{21}=a_{22}=1$. It was proved that
the two species are segregated for all times if they do so initially. Here,
segregation means that the intersection of the supports of the densities 
is empty.

\begin{figure}[htb]
\includegraphics[width=80mm]{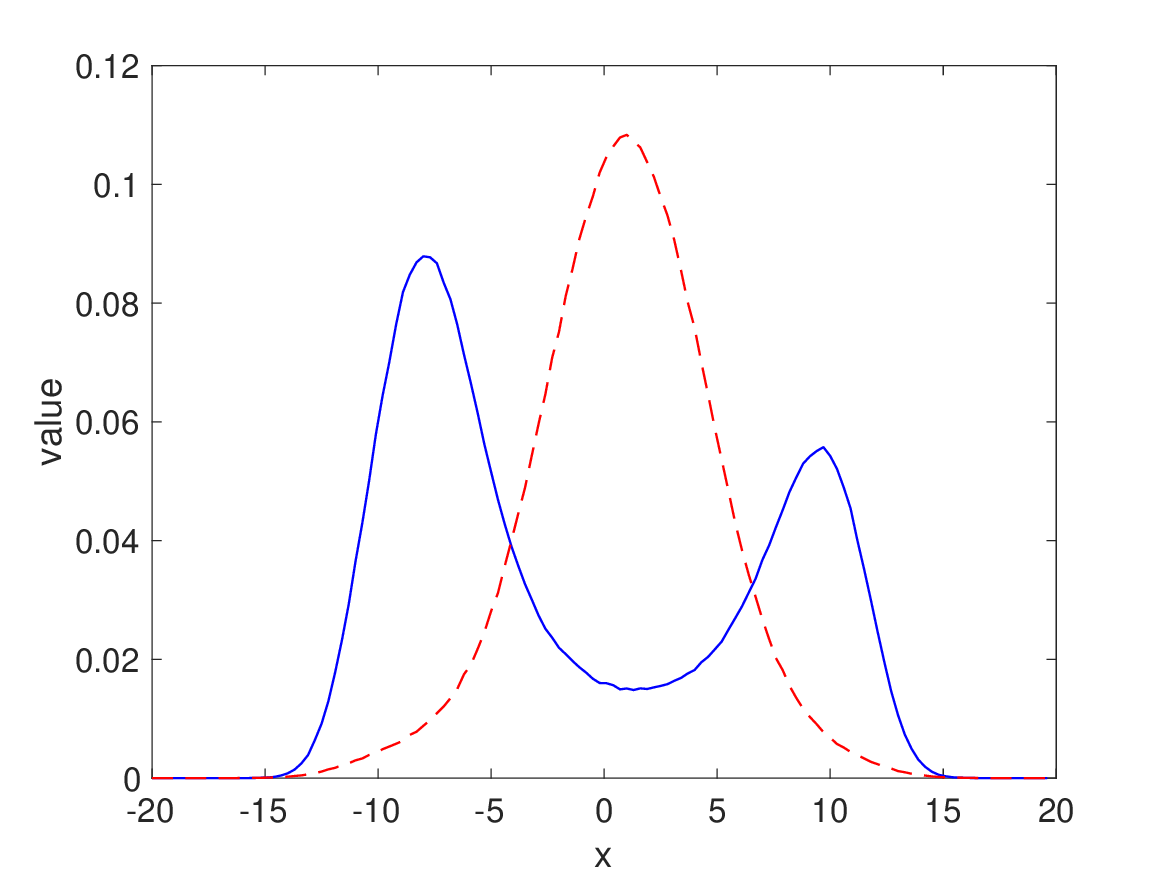}
\includegraphics[width=80mm]{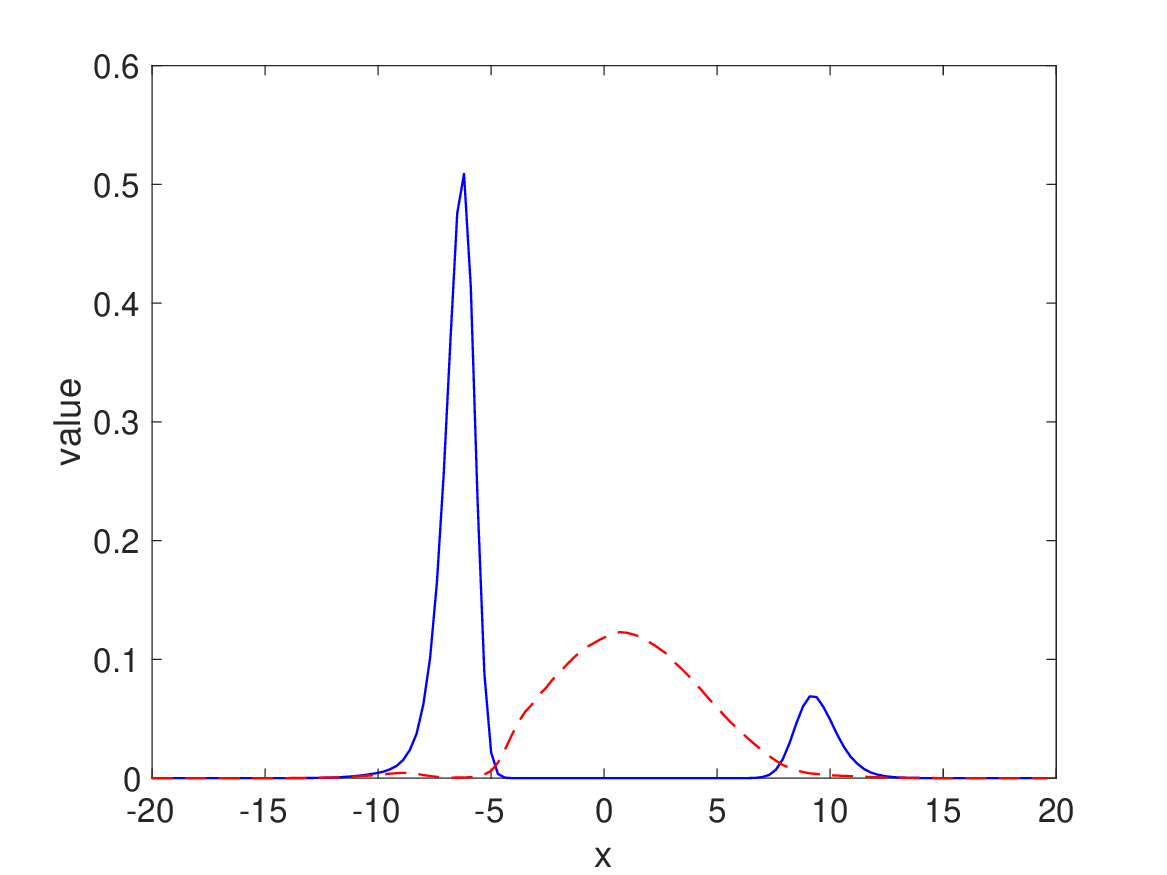}
\caption{Nonsymmetric case: 
Densities of particle system \eqref{1.pl} corresponding to the SKT population 
model (left) and particle system \eqref{1.cdj2} (right) at time $t=2$. 
Solid blue line: species 1; Dashed red line: species 2.}
\label{fig.nsymm}
\end{figure}

\subsection{Two species: symmetric case}\label{symm_2species}

We investigate the symmetric case by choosing $a_{11}=a_{22}=0$,
$a_{12}=a_{21}=355$, and, as before, $\sigma_1=1$, $\sigma_2=2$. 
The initial data are chosen as in the previous example.
In this example, we expect that cross-diffusion 
dominates self-diffusion. We present the approximate densities for different times
in Figure \ref{fig.symm}. In both models, the species have the tendency to
segregate. 
As expected, the segregation in the 
particle system \eqref{1.cdj2} is stronger than in system \eqref{1.pl}
corresponding to the SKT model.

\begin{figure}[htb]
\includegraphics[width=65mm]{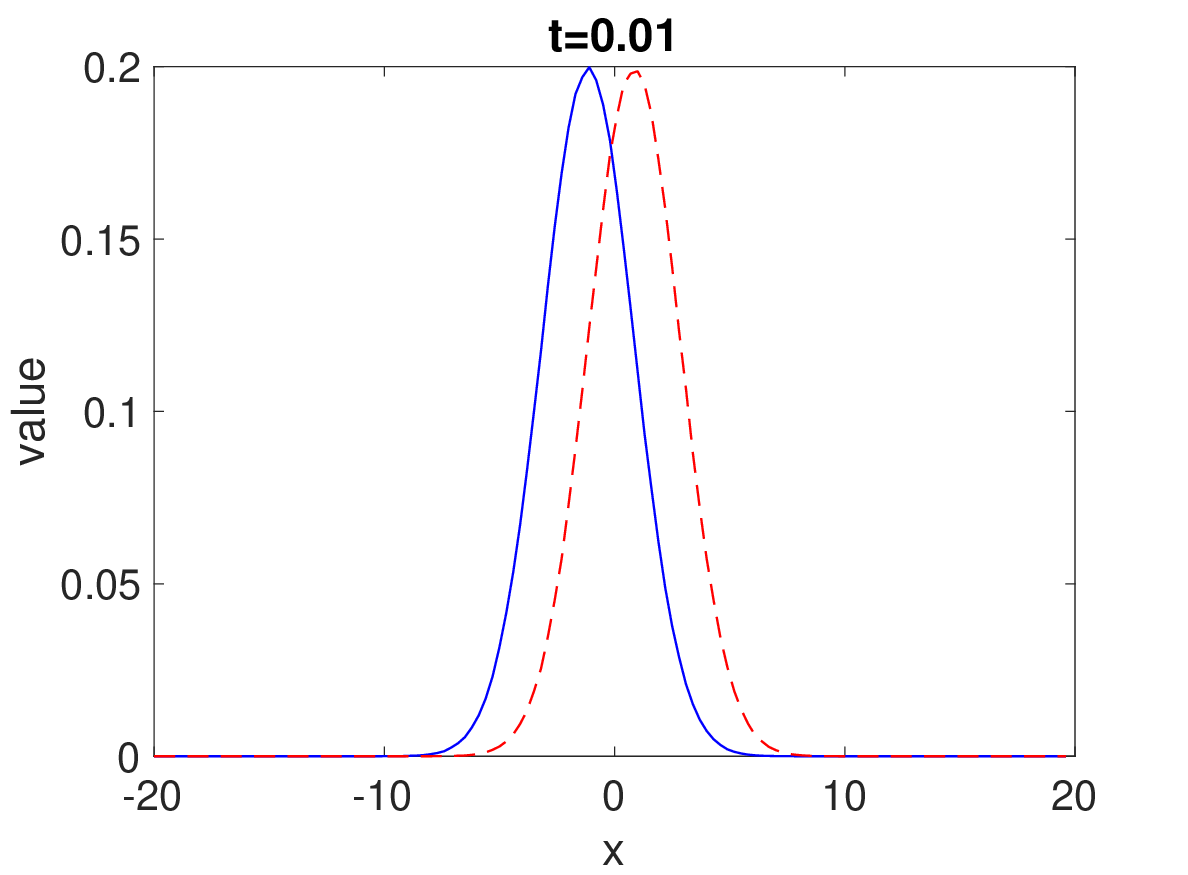} 
\includegraphics[width=65mm]{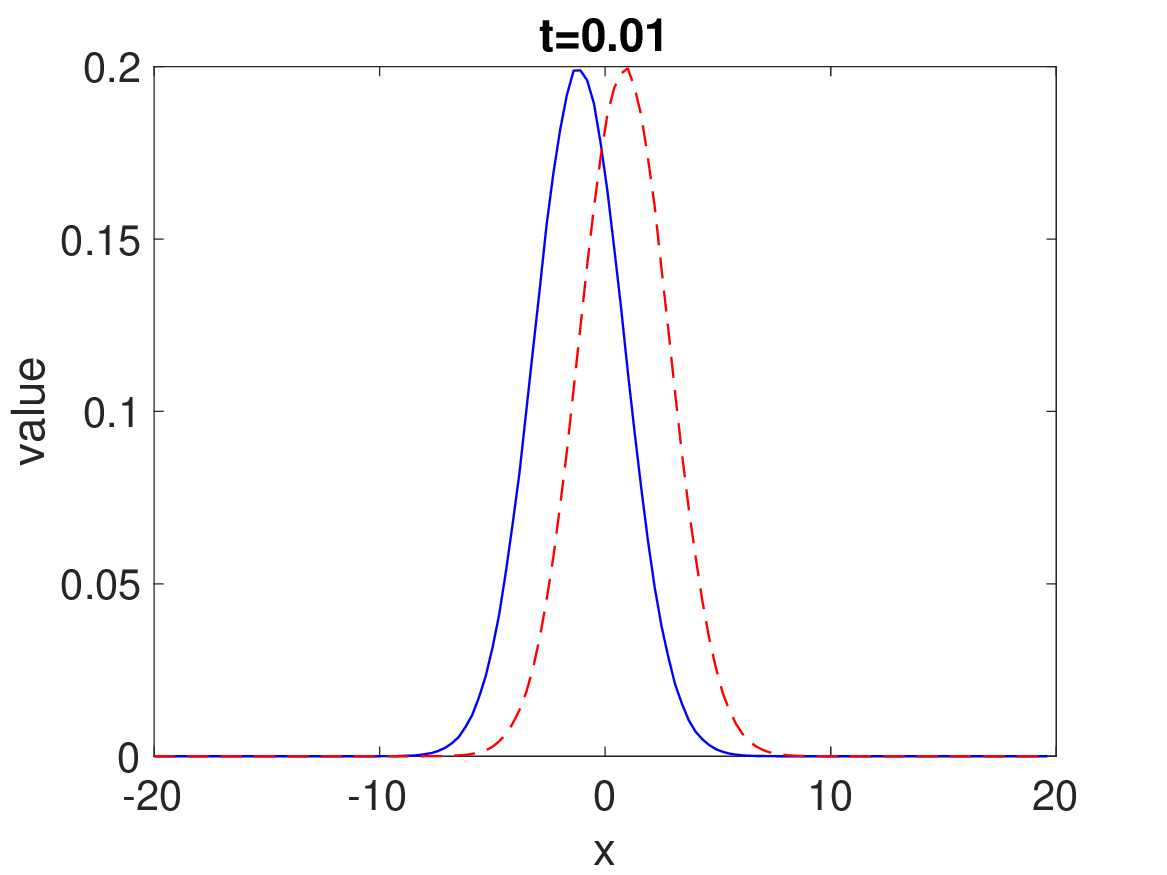} 
\includegraphics[width=65mm]{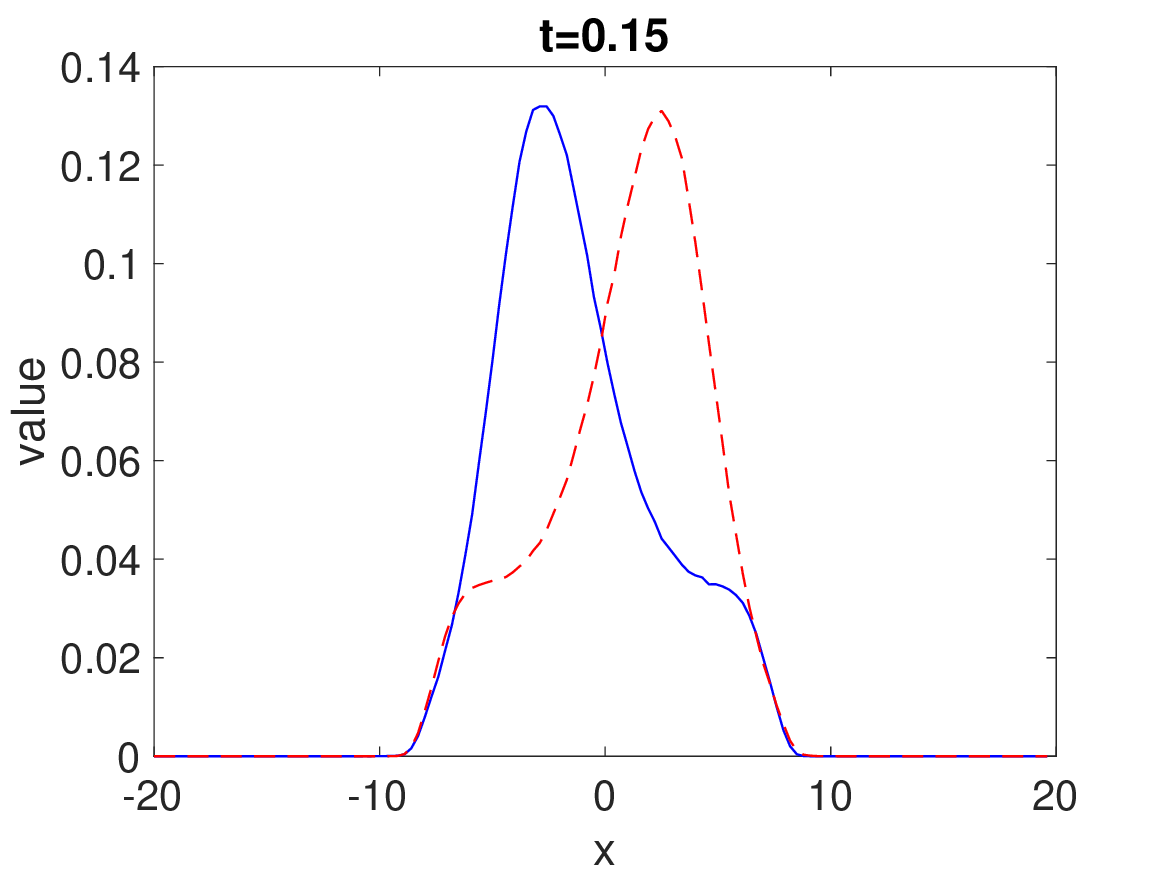} 
\includegraphics[width=65mm]{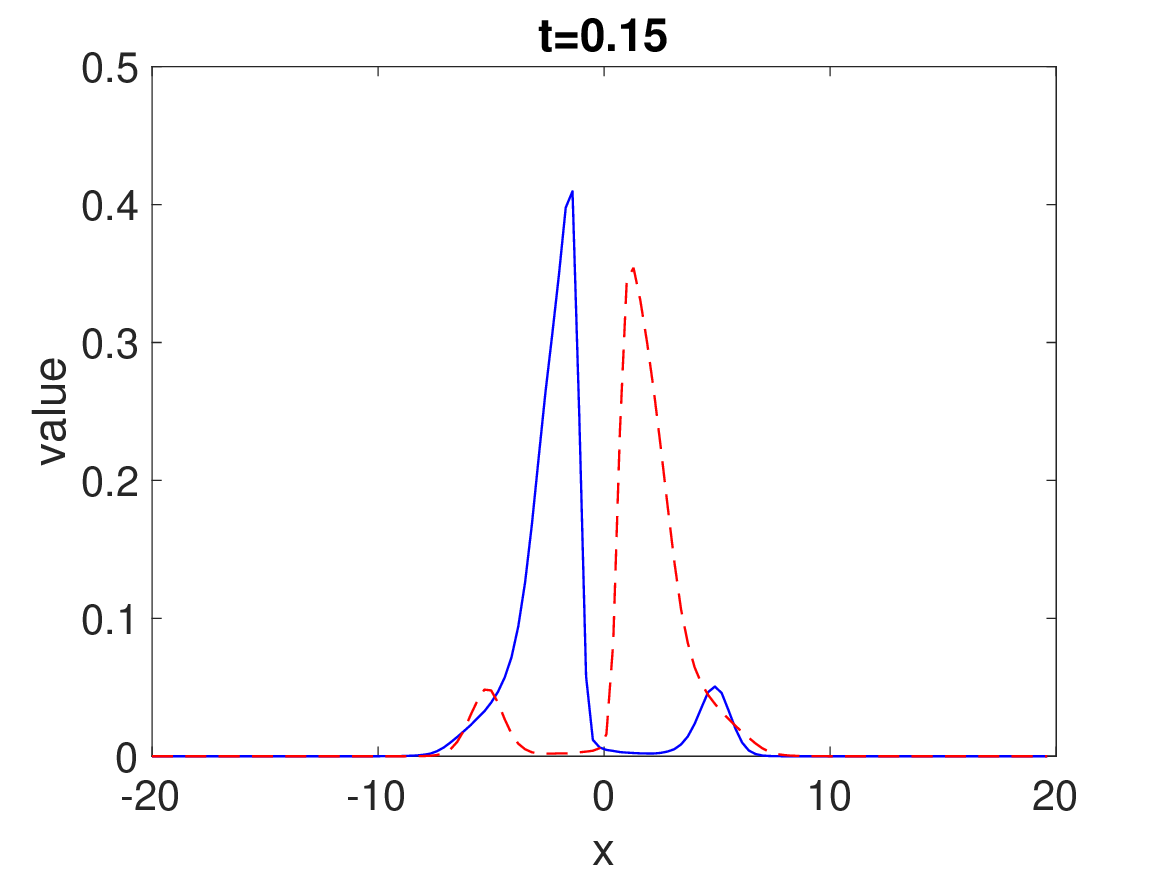} 
\includegraphics[width=65mm]{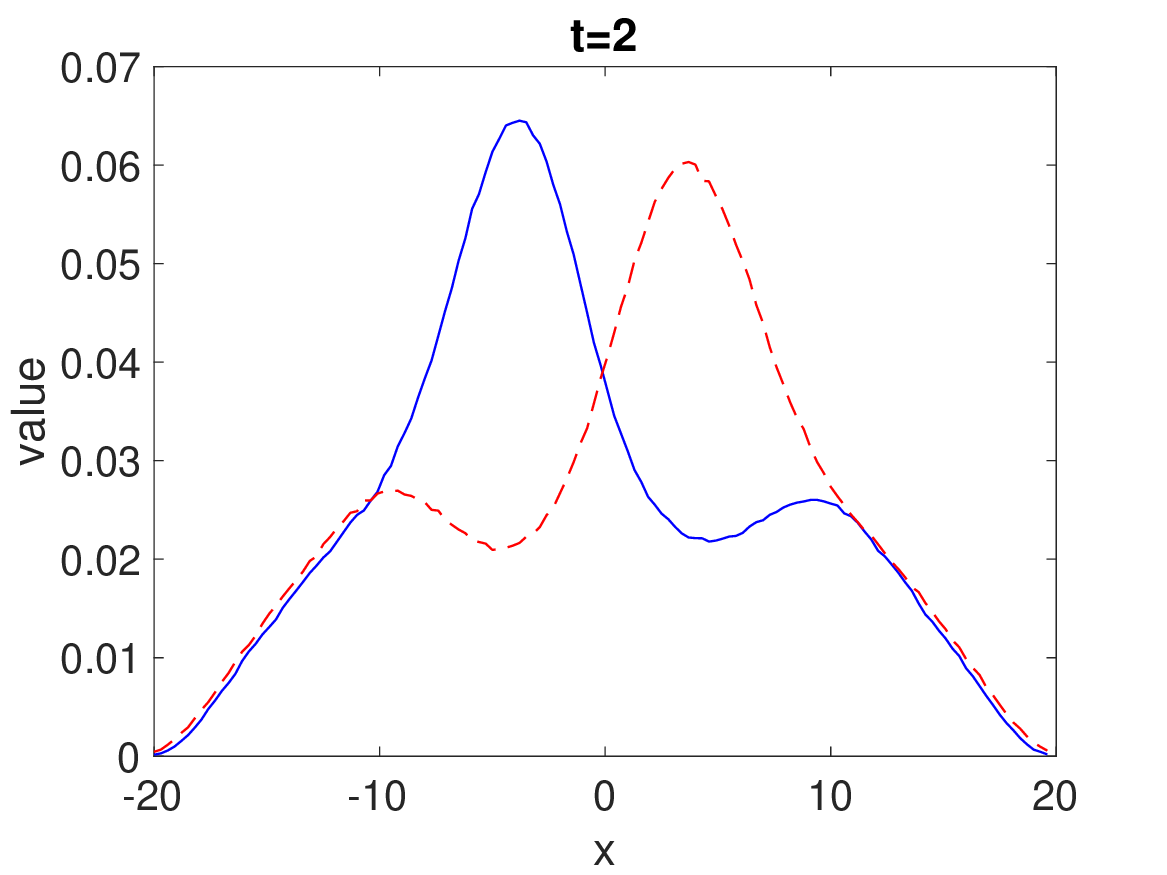}
\includegraphics[width=65mm]{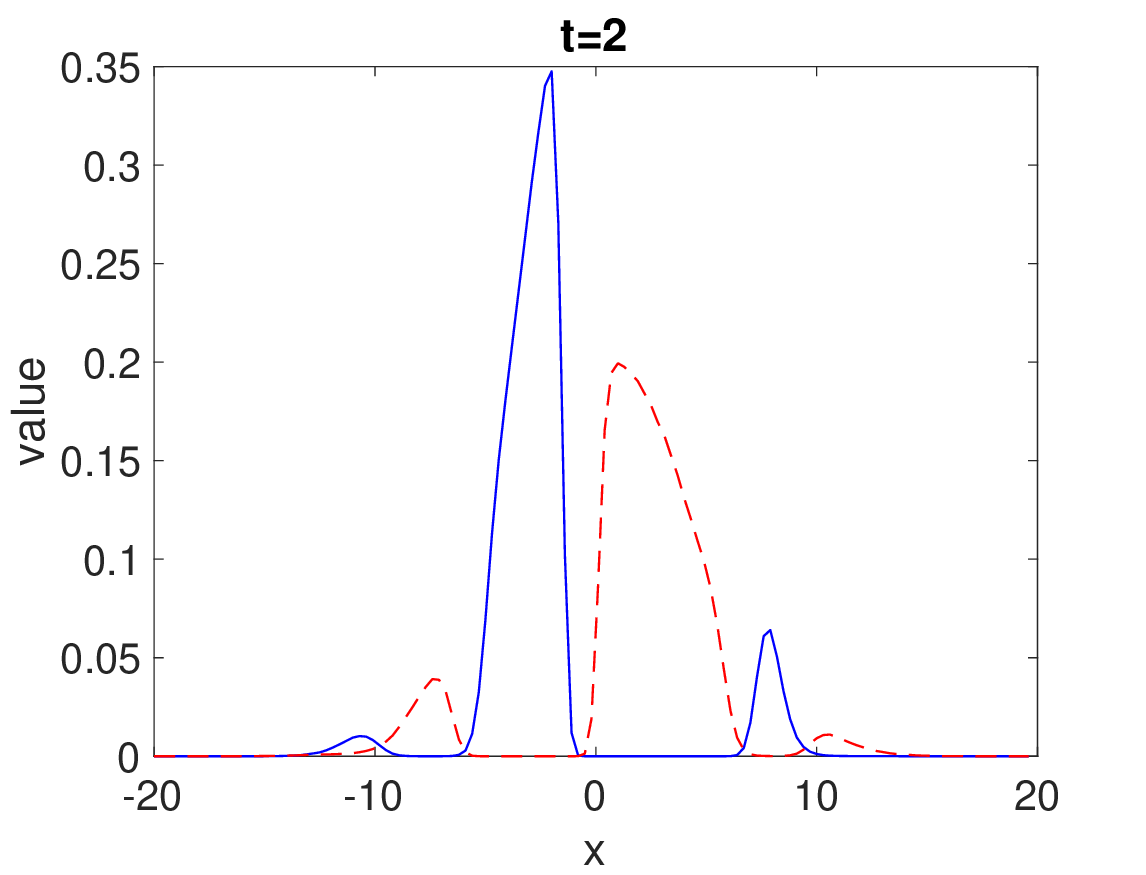}
\caption{Symmetric case: Densities of particle system \eqref{1.pl}
corresponding to the SKT population model (left) and
particle system \eqref{1.cdj2} (right) for different times $t=0.01$, $0.15$, $2$. 
Solid blue line: species 1; dashed red line: species 2.}
\label{fig.symm}
\end{figure}

\subsection{Three species}

Our third numerical experiment illustrates the segregation behaviour in case of 
three interacting species with coefficients $\sigma_1=1$, $\sigma_2=2$, $\sigma_3=3$
and 
$$
  (a_{ij}) = \begin{pmatrix} 0 & 355 & 355 \\ 25 & 0 & 25 \\ 355 & 0 & 0 \end{pmatrix}.
$$
Similar as in the two-species case, the initial data are overlapping normal 
distributions with means $-1$, 2, and $-3$, respectively, and variance 2.
The approximate densities at $t=2$ are shown 
in Figure \ref{fig.3species}. We observe that the approximate densities of particle 
model \eqref{1.cdj2} show a much clearer component-wise segregation behavior than 
the stochastic particle model \eqref{1.pl}, which corresponds to the SKT system,
where the diffusion effects are much stronger. This may be explained by the fact
that, on the PDE level, the gradient-flow structure of model \eqref{1.cdj} can
be written species-wise, whereas the SKT model \eqref{1.skt} (with $f(x)=x$)
only posseses a vector-valued gradient-flow structure.

\begin{figure}[htb]
\includegraphics[width=80mm]{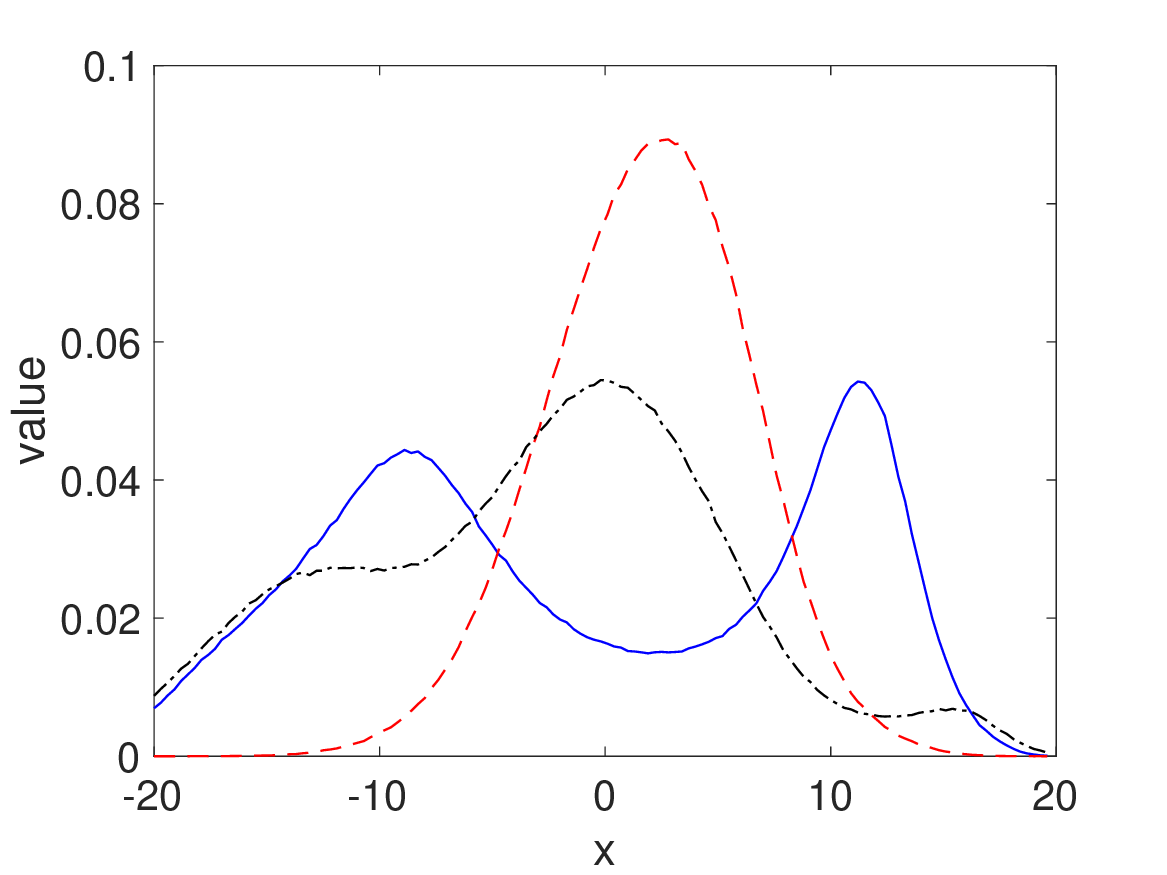}
\includegraphics[width=80mm]{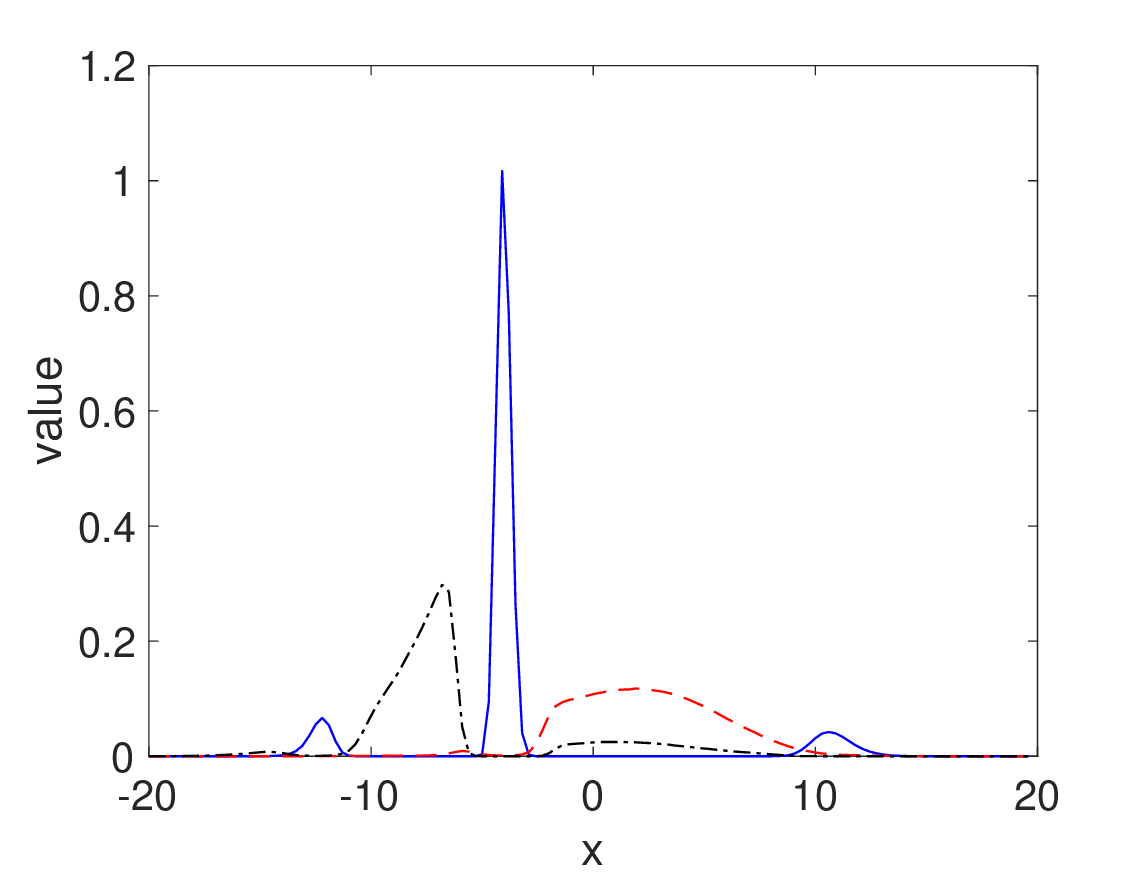}
\caption{Three-species case:
Densities of particle system \eqref{1.pl} corresponding to the SKT population model 
(left) and particle system \eqref{1.cdj2} (right) at time $t=2$. 
Solid blue line: species 1; dashed red line: species 2; dash-dotted black line: 
species 3.}
\label{fig.3species}
\end{figure}

\subsection{Cubic nonlinearity}\label{cubic}

\red{For our last experiment, we compare the numerical results for the cubic
nonlinearity $f(s)=s^3$ with the linear case imposed in the previous examples.
The parameters are the same as in Section \ref{symm_2species}. The numerical
simulations are performed without using approximating functions $f_{\eta}$. 
This may be justified by the fact that the simulations deal with
(relatively) small time scales and with compactly supported initial data.
We observe in Figure \ref{fig.cubic} that the cubic nonlinearity causes more 
clustering than the linear case $f(s)=s$. The simulations suggests that in the cubic 
case, diffusion happens on a faster time scale than segregation, while in the
linear case, the particles diffuse slower and hence they form bigger but fewer 
clusters.}      

\begin{figure}[htb]
\includegraphics[width=80mm]{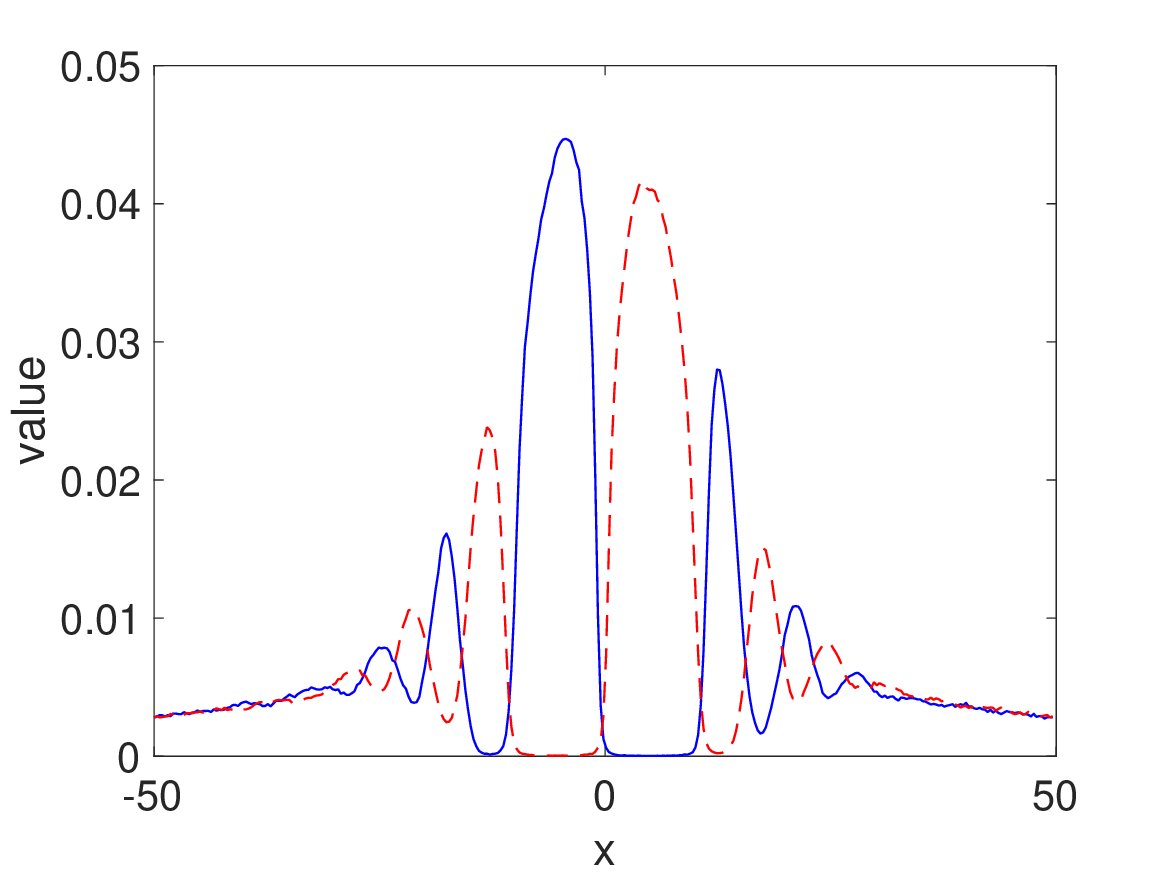}
\includegraphics[width=80mm]{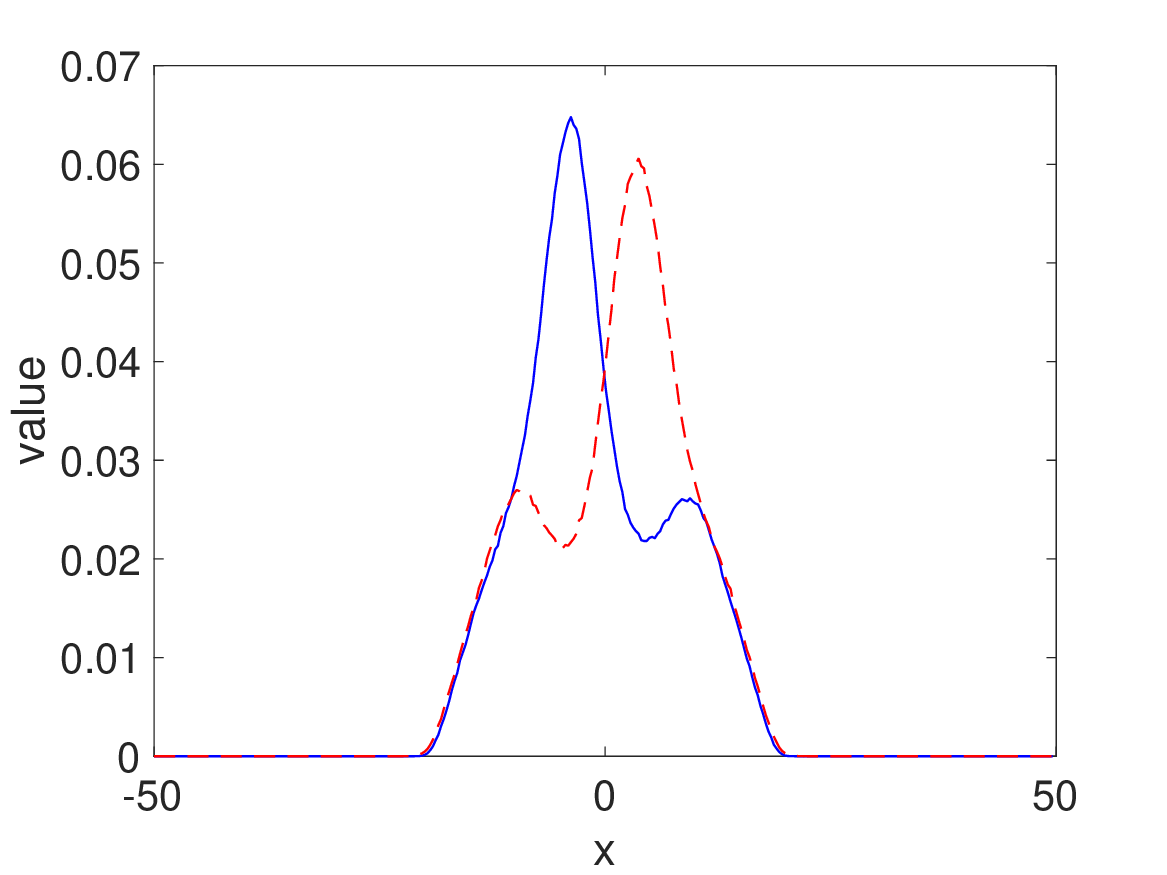}
\caption{Densities of particle system \eqref{1.pl} corresponding to the SKT 
population model with $f(s)=s^3$ (left) and $f(s)=s$ (right) at time $t=2$. 
Solid blue line: species 1; dashed red line: species 2.
The right figure is the same as in Figure \ref{fig.symm} (bottom right) 
but with the range $x=-50,\ldots,50$.}
\label{fig.cubic}
\end{figure}


\begin{appendix}
\section{Auxiliary results}

For the convenience of the reader, we recall some well-known estimates 
used in this paper. 

\begin{lemma}[Young's convolution inequality, {\cite[Formula (7), page 107]{LL01}}]
\label{lem.convul}
Let $1\le p,q,r$ $\le\infty$ be such that $1/p+1/q=1+1/r$ and let $f\in L^p(\R^d)$,
$g\in L^q(\R^d)$. Then $f*g\in L^r(\R^d)$ and
$$
  \|f*g\|_{L^r(\R^d)}\le \|f\|_{L^p(\R^d)}\|g\|_{L^q(\R^d)}.
$$
\end{lemma}

\begin{lemma}[Moser-type estimate I, {\cite[Prop.~2.1(A)]{Maj84}}]\label{lem.moser}
Let $s\in\N$ and $\alpha\in\N_0^n$ with $|\alpha|=s$. Then there exists
a constant $C>0$ such that for all $f$, $g\in H^s(\R^d)\cap L^\infty(\R^d)$,
$$
  \|D^\alpha(fg)\|_{L^2(\R^d)} \le C\big(\|f\|_{L^\infty(\R^d)}\|D^s g\|_{L^2(\R^d)}
	+ \|D^s f\|_{L^2(\R^d)}\|g\|_{L^\infty(\R^d)}\big).
$$
\end{lemma} 

\begin{lemma}[Moser-type estimate II, {\cite[Prop.~2.1(C)]{Maj84}}]\label{lem.moser2}
Let $s\in\N$ and $\alpha\in\N_0^n$ with $|\alpha|=s$. Then there exists
a constant $C>0$ such that for smooth $g:\R\to\R$ and 
$u\in H^s(\R^d)\cap L^\infty(\R^d)$,
$$
  \|D^\alpha g(u)\|_{L^2(\R^d)} \le C\|g'\|_{C^{s-1}(\R)}\|u\|_{L^\infty(\R^d)}^{s-1}
	\|D^\alpha u\|_{L^2(\R^d)}.
$$
\end{lemma}

\begin{lemma}[Moser-type commutator inequality, {\cite[Prop.~2.1(B)]{Maj84}}]
\label{lem.comm}
Let $s\in\N$ and $\alpha\in\N_0^n$ with $|\alpha|=s$. Then
there exists $C>0$ such that for all $f\in H^s(\R^d)\cap W^{1,\infty}(\R^d)$
and $g\in H^{s-1}(\R^d)\cap L^\infty(\R^d)$,
$$
  \|D^\alpha(fg)-fD^\alpha(g)\|_{L^2(\R^d)} \le C\big(\|Df\|_{L^\infty(\R^d)}
	\|D^{s-1}g\|_{L^2(\R^d)} + \|D^s f\|_{L^2(\R^d)}\|g\|_{L^\infty(\R^d)}\big),
$$
where $D^s=\sum_{|\alpha|=s}D^\alpha$.
\end{lemma}

\end{appendix}


\end{document}